\documentclass[10pt,english]{smfart}

\usepackage[T1]{fontenc}
\usepackage[english,french]{babel}
\usepackage{latexsym,amscd,color}
\usepackage{amsmath,amsfonts,amssymb,mathrsfs}
\usepackage{enumerate,euscript}
\usepackage{amssymb,url,xspace,smfthm}
\input xy
\xyoption{all}

\DeclareMathOperator{\codim}{codim}
\DeclareMathOperator{\gal}{Gal}

\newcommand{\BibTeX}{{\scshape Bib}\kern-.08em\TeX}

\newcommand{\T}{\S\kern .15em\relax }
\newcommand{\AMS}{$\mathcal{A}$\kern-.1667em\lower.5ex\hbox
        {$\mathcal{M}$}\kern-.125em$\mathcal{S}$}

\DeclareMathOperator{\vol}{vol}
\DeclareMathOperator{\im}{Im}

\DeclareMathOperator{\spm}{Spm}

\DeclareMathOperator{\pr}{pr}
\DeclareMathOperator{\rg}{rk}

\DeclareMathOperator{\spec}{Spec}

\renewcommand{\P}{\mathbb{P}}
\newcommand{\wmu}{\widehat{\mu}}
\newcommand{\C}{\mathbb{C}}
\newcommand{\Q}{\mathbb{Q}}

\newcommand{\adeg}{\widehat{\deg}}

\newcommand{\p}{\mathfrak{p}}
\DeclareMathOperator{\sym}{Sym}

\renewcommand{\H}{\overline{H}}
\newcommand{\E}{\overline{E}}
\newcommand{\F}{\overline{F}}
\newcommand{\sE}{\mathcal{E}}
\newcommand{\G}{\overline{G}}

\renewcommand{\O}{\mathcal{O}}
\renewcommand{\L}{\overline{L}}
\newcommand{\sm}{\mathfrak{m}}
\newcommand{\sn}{\mathfrak{n}}

\newcommand{\f}{\mathbb{F}}
\newcommand{\ndot}{\raisebox{.4ex}{.}}

\title{Determinant method and the pseudo-effective threshold}
\alttitle{La m\'ethode de determinant et le seuil de pseudo-effectivit\'e}

\date{\today}
\author{Chunhui Liu}
\address{Institute for Advanced Study in Mathematics\\
Harbin Institute of Technology\\
150001 Harbin\\P. R. China}
\email{chunhui.liu@hit.edu.cn}
\begin{document}
\def\smfbyname{}
\begin{abstract}
In this paper, we will give an upper bound of the number of auxiliary hypersurfaces in the determinant method, which reformulates an unpublished work of Salberger by Arakelov geometry. One of the key constants will be determined by the pseudo-effective threshold of certain line bundles.
\end{abstract}
\begin{altabstract}
Dans cet article, on donnera une majoration du nombre de hypersurfaces auxiliaires dans la m\'ethodede d\'eterminant, qui reformule un travail non publi\'e de Salberger par la g\'eom\'etrie d'Arakelov. Une des constantes cl\'ees sera d\'etermin\'ee par le seuill de pseudo-effectivit\'e de certains fibr\'es en droites.
\end{altabstract}

\maketitle

\tableofcontents
\section{Introduction}
Let $K$ be a number field, and $X\hookrightarrow\mathbb P^n_K$ be a projective variety. Let $\xi\in X(K)$, and $H_K(\xi)$ be a height of $\xi$ with respect to the above closed immersion, for example, the classic Weil height (cf. \cite[\S B.2, Definition]{Hindry}). A height function $H_K(\ndot)$ on the set of rational points is able to be used to measure their arithmetic complexities. Let $B\in\mathbb R$, and
\[S(X;B)=\{\xi\in X(K)|\;H_K(\xi)\leqslant B\}\]
be the set of rational points of bounded heights with respect to the above closed immersion. Usually, a good height function has the so-called Northcott's property, which means that the cardinality $N(X;B)=\#S(X;B)$ is finite when $B$ is fixed. In this case, the map $N(X;\cdot):\;\mathbb R\rightarrow \mathbb N$ is a function which gives a description of the density of rational points in $X$.

It is a central subject to understand different kinds of properties of the function $N(X;B)$ with the variable $B\in\mathbb R$ for different kinds of $X$. For this target, lots of methods have been developed. In this article, we will focus on the uniform upper bound of $N(X;B)$. The word "uniform" means that we want to obtain a good upper bound of $N(X;B)$ for a family of projective varieties satisfying certain common conditions, for example, with the same degree and dimension.

\subsection{Determinant method}
In this article, we will focus on the so-called determinant method proposed in \cite{Heath-Brown} to study the density of rational points in arithmetic varieties.
\subsubsection{Basic ideas and the developments}
Tranditionally, we consider a projective variety $X\hookrightarrow\mathbb P^n_{\Q}$ over $\Q$ for simplicity, since the operations over arbitrary number fields sometimes bring us extra technical troubles. In \cite{Bombieri_Pila} (see also \cite{Pila95}), Bombieri and Pila proposed a method of determinant argument to study plane affine curves. In \cite{Heath-Brown}, Heath-Brown developed the so-called the \textit{$p$-adic determinant method}, which generalized the method of \cite{Bombieri_Pila} to the higher dimensional case. His idea is to focus on a subset of $S(X;B)$ whose reductions modulo a prime number are a same regular point, and he proved that this subset can be covered by a bounded degree hypersurface which do not contain the generic point of $X$.  By Siegel's Lemma, we can assure the existence of such hypersurfaces in $\mathbb P^n_{\Q}$ with bounded degree. Then he counted the number of regular points over finite fields, and control the regular reductions. By this method, he proved that $N(X;B)\ll_{d,\delta,\epsilon}B^{\frac{2}{\delta}+\epsilon}$ for all $\epsilon>0$, where $\delta=\deg(X)$. In \cite{Broberg04}, Broberg generalized it to the case over an arbitrary number field.

In \cite{Heath-Brown}, Heath-Brown also proposed a so-called \textit{the dimension growth conjecture}. Let $\dim(X)=d$. It is said that for all $d\geqslant2$ and $\delta\geqslant2$, we have $N(X;B)\ll_{d,\delta,\epsilon}B^{d+\epsilon}$ for all $\epsilon>0$. He proved this conjecture for some special cases. Later, Browning, Heath-Brown and Salberger had some contributions on this subject, see \cite{Browning_Heath06I,Browning_Heath06II,Bro_HeathB_Salb,Salberger07,Salberger_preprint2013} for the refinements of the determinant method and the proofs under certain conditions.

%In \cite{Salberger07}, Salberger considered the general reductions, and the multiplicities of rational points are taken into consideration, and he proved the dimension growth conjecture with certain conditions on the sub-varieties of $X$. In his brilliant work \cite{Salberger_preprint2013}, he developed a so-called the global determinant method, which makes sure that when $X$ is a geometrically integral hypersurface in $\mathbb P^n_K$, the set $S(X;B)$ can be covered by one hypersurface with proper bounded degree. By the global version, he proved the conjecture for $\delta\geqslant4$ and $N(X:B)\ll_{d,\delta}B^{\frac{2}{\delta}}\log B$ when $X$ is a curve.

%Then he prove this estimate for arbitrary hypersurfaces. By this refinement, he proved $N(X:B)\ll_{d,\delta}B^{\frac{2}{\delta}}$ when $X$ is a curve, which is possibly optimal.

%In fact, they gave an explicit dependence on some constants on $\delta$, and proved $N(X;B)\ll_n \delta^{e_1}B^d$ for $\delta\geqslant5$, where $e_1$ is an explicit constant depending on $n$ only, and $N(X;B)\ll_{n}\delta^4B^{\frac{2}{\delta}}$ when $X$ is a plane curve.

In \cite{Salberger2015}, Salberger considered the case of cubic hypersurfaces, where we have a better estimate on a key invariant than that was obtained in \cite{Bro_HeathB_Salb,Salberger_preprint2013}. Actually, this work essentially applied the refinement of the invariant mentioned above by the pseudo-effective thresholds of certain line bundles.

\subsubsection{Reformulation by Arakelov geometry}
In \cite{Chen1,Chen2}, H. Chen reformulated the works of Salberger \cite{Salberger07} by the slope method in Arakelov geometry. By this formulation, we replace the matrix of monomials by the evaluation map which sends a global section of a particular line bundle to a family of rational points. By the slope inequalities, we can control the height of the evaluation map in the slope method, which replaces the role of Siegel's lemma in controlling heights.

There are two advantages by the approach of Arakelov geometry. First, Arakelov geometry gives a natural conceptual framework for the determinant method over an arbitrary number field. Second, it is easier to obtain explicit estimates, since usually the constants obtained by the slope method are given explicitly.

But in this article, because of certain obstructions in the study of the positivity of line bundles, we are not able to give effective estimates for all invariants. We will explain the exact reason later.

\subsection{Application of the pseudo-effective threshold}
 In a mini-course of the summer school "Arakelov Geometry and Diophantine applications" at Institut Fourier in 2017, and a mini-course of the thematic activity "Reinventing rational points" at Institut Henri Poincar\'e in 2019, Salberger gave lectures on the application of the pseudo-effective thresholds of certain line bundles on projective varieties to estimate the number of auxiliary hypersurfaces in the determinant method. In \cite{Salberger2015}, he has applied this idea to study the density of rational points in the complement of the union of all lines of cubic surfaces in $\mathbb P^3$.

 In this article, we will reformulate the above works of Salberger by Arakelov geometry following the strategy of \cite{Chen1, Chen2}, where we will consider the case of general projective varieties. Some ideas of this work has been applied in \cite{Salberger2015}.
\subsubsection{Role of pseudo-effective threshold}
Let $X\hookrightarrow\mathbb P^n_K$ be a projective variety over the number field $K$ of degree $\delta$ and dimension $d$, $\pi:\;\widetilde{X}\rightarrow X$ be the blowing up at the non-singular rational point $\eta$, $E$ is the exceptional divisor of this blowing up, $H$ be a Cartier divisor on $X$ given by a hyperplane section on $\mathbb P^n_K$, and $D,m\in\mathbb N$. We consider the sum
\begin{equation}\label{R(E)_introduction}
  R(\eta,D)=\sum_{m=1}^\infty \dim_K H^0\left(\widetilde{X},D\pi^*H-mE\right),
\end{equation}
which plays a significant role in Salberger's refinement of $p$-adic determinant mentioned above. Next, we denote
\begin{equation}\label{I_X_introduction}
  I_X(H,\eta)=\int_0^\infty \vol (\pi^*H-\lambda E)d\lambda,
\end{equation}
 where $\vol(\ndot)$ is the usual volume function of $\mathbb R$-divisors. In Theorem \ref{R(E) by GIT}, we will give a proof of the estimate
\begin{equation}\label{R(E)_introduction2}
  R(\eta,D)=\frac{I_X(H,\eta)}{d!}D^{d+1}+O_{d,\delta}(D^d).
\end{equation}
By this fact, we can refine some former results on the determinant method.
\subsubsection{An improved upper bound of the number of auxiliary hypersurfaces}
Let $\mathscr X\hookrightarrow\mathbb P^n_{\O_K}$ be the Zariski closure of $X\hookrightarrow\mathbb P^n_K$, and $\p$ be a maximal ideal of $\O_K$ whose residue field is $\f_\p$. Let $\xi\in\mathscr X(\f_\p)$, and we denote by $S(X;B,\xi)$ the subset of $S(X;B)$ the reduction modulo $\p$ of whose Zariski closures in $\mathscr X$ is $\xi$. We can prove that the invariant $I_X(H,\eta)$ only depends on its reduction class if its reduction is regular. By Lemma \ref{I_X(H,xi) same}, if for the family of maximal ideals $\p_1,\ldots,\p_r$ of $\O_K$, the point $\xi_j$ is regular in $\mathscr X$ for all $j=1,\ldots,r$ and $\bigcap\limits_{j=1}^rS(X;B,\xi_j)\neq\emptyset$, then all $I_X(H,\xi_j)$ are equal, noted by $I_X(H,\xi_J)$ for simplicity. Then we have the result below from \eqref{R(E)_introduction2}, where Salberger has proved the case of $K=\Q$.
\begin{theo}[Theorem \ref{semi-global determinant method mid}]\label{semi-golbal determinant method introduction}
  We keep all the above notations. Let $\p_1,\ldots,\p_r$ be a family of maximal ideals of $\O_K$, $N(\p_j)=\#\left(\O_K/\p_j\right)$, and $\epsilon>0$. Suppose that the point $\xi_j\in\mathscr X(\f_{\p_j})$ is regular in $\mathscr X$ for all $j=1,\ldots,r$. If the inequality
  \[\sum_{j=1}^r\log N(\p_j)\gg_{K,n,\delta,\epsilon}\frac{\delta}{I_X(H,\xi_J)}\log B\]
  is verified, then there exists a hypersurface of degree $O_{d,\delta,\epsilon}(1)$, which covers $\bigcap\limits_{j=1}^rS(X;B,\xi_j)$ but do not contain the generic point of $X$.
\end{theo}
 By this result, let $\epsilon>0$ and
\[I_X(H)=\inf\limits_{\begin{subarray}{c} \eta\in S(X;B)\\ \eta\hbox{ regular}\end{subarray}}I_X(H,\eta).\]
Then we have the following estimate of the number of auxiliary hypersurfaces from Theorem \ref{semi-golbal determinant method introduction}, where Salberger has proved the case of $K=\Q$, too.
 \begin{theo}[Theorem \ref{number of hypersurfaces}]\label{number of auxiliary hypersurface_introduction}
  With all the notations above. There exists a constant $C_4(\epsilon, \delta, n,d, K)$ such that $S(X;B)$ is covered by no more than
 \begin{equation}\label{number of hypersurfaces_introduction}
  C_4(\epsilon, \delta, n,d, K)B^{\frac{(1+\epsilon)d\delta}{I_X(H)}}
 \end{equation}
  hypersurfaces of degree $O_{n,\delta,\epsilon}(1)$ which do not contain the generic point of $X$.
 \end{theo}

By an unpublished result of Salberger (see also \cite[Corollary 4.2]{McKinnonRoth_2015}), for every regular closed point $\eta$ in $X$, we have $I_X(H,\eta)\geqslant{d\delta^{1+\frac{1}{d}}}/{(d+1)}$. In this sense, the upper bound of the number of auxiliary hypersurfaces given in \eqref{number of hypersurfaces_introduction} can be considered as an improvement of some former results (\cite{Heath-Brown, Salberger07, Chen2}, for example). If we focus on some special varieties $X$ with clearer information on $I_X(H,\eta)$ defined at \eqref{I_X_introduction}, we may obtain a better estimate on the number of auxiliary hypersurfaces, see \cite{Salberger2015} for such an example, where the case of cubic hypersurfaces in $\mathbb P^3$ is considered.

\subsubsection{Ineffective estimates}
In the above argument, we have
\[\dim_KH^0\left(\widetilde{X},D\pi^*H-mE\right)=\frac{D^d}{d!}\vol\left(\pi^*H-\frac{m}{D}E\right)+O_{d,\delta}(D^{d-1}).\]
However, up to the author's knowledge, we are not able to obtain an effective version in the above estimate. Thus we are only able to make sure that the maximal degree of auxiliary hypersurfaces can depend only on $n$, $\delta$ and $\epsilon$, but we cannot get an explicit bound until now.
\subsection{Organization of the article}
This article is organized as follows. In \S2, we will recall some useful preliminaries and propose the basic setting, where we follow the approach of \cite{Chen1,Chen2}. In \S3, we will give a bound relating to the invariant $R(\eta,D)$ defined in \eqref{R(E)_introduction} and both geometric and arithmetic Hilbert-Samuel functions of arithmetic varieties, which is a generalization of \cite[Lemma 16.9]{Salberger2015}. In \S4, we will prove the finiteness of the sum \eqref{R(E)_introduction} and the asymptotic estimate \eqref{R(E)_introduction2}, which reformulates some former results of Salberger. In \S5, we will prove Theorem \ref{semi-golbal determinant method introduction}, and give the upper bound \eqref{number of hypersurfaces_introduction} in Theorem \ref{number of auxiliary hypersurface_introduction} by applying it.
\subsection*{Acknowledgement}
The author learnt some unpublished key results of Prof. Per Salberger from his mini-course in the summer school "Arakelov Geometry and Diophantine applications" at Institut Fourier in 2017, and from his mini-course in the thematic activity "Reinventing rational points" at Institut Henri Poincar\'e in 2019, and these courses motivate this article. The author would like to thank Prof. Salberger for introducing him his brilliant work \cite{Salberger_preprint2013} and some useful personal notes, and also for lots of useful private discussion. At the same time, The author would like to thank Prof. Yuji Odaka for some useful suggestions on the pseudo-effective thresholds.

\section{Preliminaries and the basic setting}
In this section, we will provide some preliminaries that will be used to interpret the determinant method in terms of Arakelov theory, where we follow the strategy of H. Chen in \cite{Chen1,Chen2}.
\subsection{Classic height function of rational points}
Let $K$ be a number field, and $\O_K$ be its ring of integers. We denote by $M_{K,f}$ the set of finite places of $K$, and by $M_{K,\infty}$ the set of infinite places of $K$. In addition, we denote by $M_K=M_{K,f}\sqcup M_{K,\infty}$ the set of places of $K$. For every $v\in M_{K,f}$, if $\Q_v$ is the $p$-adic field, we define the absolute value $|x|_v=\left|N_{K_v/\Q_v}(x)\right|_p^\frac{1}{[K_v:\Q_v]}$, where $|\ndot|_p$ is the usual $p$-adic absolute value. For every $v\in M_{K,\infty}$, we define $|x|_v=\left|N_{K_v/\Q_v}(x)\right|^\frac{1}{[K_v:\Q_v]}$, where $|\ndot|$ is the usual absolute values over $\mathbb R$ or $\mathbb C$.

For every $a\in K^\times$, we have the \textit{product formula} (cf. \cite[Chap. III, Proposition 1.3]{Neukirch})
\begin{equation}\label{product formula}
\prod_{v\in M_K}|a|_v^{[K_v:\Q_v]}=1.
\end{equation}

 Let $\xi=[\xi_0:\cdots:\xi_n]\in\mathbb P^n_K(K)$. We define the \textit{absolute height} of $\xi$ in $\mathbb P^n_K$ as
\begin{equation}\label{classic absolute height}
  H_K(\xi)=\prod_{v\in M_K}\max_{0\leqslant i\leqslant n}\left\{|\xi_i|_v\right\}^{[K_v:\Q_v]}.
\end{equation}
Next, we define the \textit{logarithmic height} of $\xi$ as
\begin{equation}\label{log height}
  h(\xi)=\frac{1}{[K:\Q]}\log H_K(\xi),
\end{equation}
which is independent of the choice of $K$ (cf. \cite[Lemma B.2.1]{Hindry}).

Suppose $X$ is a closed integral subscheme of $\mathbb P^n_K$ of degree $\delta$ and dimension $d$, and $\phi:X\hookrightarrow\mathbb P^n_K$ is the projective embedding. For $\xi\in X(K)$, we define $H_K(\xi)=H_K\left(\phi\left(\xi\right)\right)$ for simplicity, and usually we omit the closed immersion $\phi$. Next, we define
\[S(X;B)=\{\xi\in X(K)|H_K(\xi)\leqslant B\},\hbox{ and } N(X;B)=\#S(X;B).\]
By the Northcott's property (cf. \cite[Theorem B.2.3]{Hindry}), the cardinality $N(X;B)$ is finite for a fixed real number $B\geqslant1$.

The objective of counting rational points of bounded height is to understand the function $N(X;B)$ with some particular projective varieties $X$ and real numbers $B\geqslant1$.
\subsection{Multiplicity of points in a scheme}\label{local algebra}
In this part, we will define the multiplicity of closed points in schemes induced by the local Hilbert-Samuel function. This notion will be useful in the determinant method.

Let $X$ be a Noetherian scheme of pure dimension $d$, which means all its irreducible components have the same dimension. Let $\xi$ be a closed point of $X$, $\sm_{X,\xi}$ be the maximal ideal of the local ring $\O_{X,\xi}$, and $\kappa(\xi)$ be its residue field. We define
\begin{equation}\label{local hilbert of a closed point}
  H_\xi(s)=\dim_{\kappa(\xi)}\left(\sm_{X,\xi}^s/\sm_{X,\xi}^{s+1}\right)
\end{equation}
as the \textit{local Hilbert-Samuel function} of $X$ at the closed point $\xi$ with the variable $s\in\mathbb N$, where we define $\sm_{X,\xi}^0=\O_{X,\xi}$ for simplicity. For this function, when $d\geqslant2$, we have the polynomial asymptotic extension
\[H_\xi(s)=\frac{\mu_\xi(X)}{(d-1)!}s^{d-1}+O(s^{d-2}),\]
where we define the positive integer $\mu_\xi(X)$ as the \textit{multiplicity} of point $\xi$ in $X$. If $d=1$, then $\O_{X,\xi}$ is a local Artinian ring. The multiplicity $\mu_\xi(X)$ is then defined as the length of the local ring $\O_{X,\xi}$ as a $\O_{X,\xi}$-module.

If $\O_{X,\xi}$ is a regular local ring, we say that $\xi$ is \textit{regular} in $X$. In this case we have $\mu_\xi(X)=1$. Otherwise we say that $\xi$ is \textit{singular} in $X$. If $X$ is pure dimensional and has no embedded component, then from the fact that $\xi$ is singular in $X$ by the above definition, we deduce $\mu_\xi(X)\geqslant2$ (cf. \cite[(40.6)]{Nagata62}).

We denote by $X^{\mathrm{reg}}$ the regular locus of $X$, and by $X^{\mathrm{sing}}$ the singular locus of $X$. By the semi-continuity of the multiplicity function, the singular locus $X^{\mathrm{sing}}$ is a closed subset of $X$. If $X$ is reduced and pure dimensional, the set $X^{\mathrm{reg}}$ is open dense in $X$ (cf. \cite[Corollary 8.16, Chap. II]{GTM52}).

\subsection{Normed vector bundles}
  The normed vector bundle is one of the main research objects in Arakelov geometry. Let $K$ be a number field and $\O_K$ be its ring of integers. A \textit{normed vector bundle} over $\spec\O_K$ is a pair $\E=\left(E,\left(\|\ndot\|_v\right)_{v\in M_{K,\infty}}\right)$, where:
  \begin{itemize}
    \item $E$ is a projective $\O_K$-module of finite rank;
    \item $\left(\|\ndot\|_v\right)_{v\in M_{K,\infty}}$ is a family of norms, where $\|\ndot\|_v$ is a norm over $E\otimes_{\O_K,v}\C$ which is invariant under the action of $\gal(\C/K_v)$.
  \end{itemize}

If all the norms $\left(\|\ndot\|_v\right)_{v\in M_{K,\infty}}$ are Hermitian, we say that $\E$ is a \textit{Hermitian vector bundle} over $\spec\O_K$. In particular, if $\rg_{\O_K}(E)=1$, we say that $\E$ is a \textit{Hermitian line bundle} over $\spec\O_K$.

Suppose that $F$ is a sub-$\O_K$-module of $E$. We say that $F$ is a \textit{saturated} sub-$\O_K$-module of $E$ if $E/F$ is a torsion-free $\O_K$-module.

Let $\E=\left(E,\left(\|\ndot\|_{E,v}\right)_{v\in M_{K,\infty}}\right)$ and $\F=\left(F,\left(\|\ndot\|_{F,v}\right)_{v\in M_{K,\infty}}\right)$ be two Hermitian vector bundles over $\spec\O_K$. If $F$ is a saturated sub-$\O_K$-module of $E$ and $\|\ndot\|_{F,v}$ is the restriction of $\|\ndot\|_{E,v}$ over $F\otimes_{\O_K,v}\C$ for every $v\in M_{K,\infty}$, we say that $\F$ is a \textit{sub-Hermitian vector bundle} of $\E$ over $\spec\O_K$.

We say that $\G=\left(G,\left(\|\ndot\|_{G,v}\right)_{v\in M_{K,\infty}}\right)$ is a \textit{quotient Hermitian vector bundle} of $\E$ over $\spec\O_K$, if for every $v\in M_{K,\infty}$, the module $G$ is a projective quotient $\O_K$-module of $E$ and $\|\ndot\|_{G,v}$ is the induced quotient space norm of $\|\ndot\|_{E,v}$.

For simplicity, we denote by $E_K=E\otimes_{\O_K}K$ in the remainder part of this article.

\subsection{Arakelov invariants}
 We will introduce some useful invariants in Arakelov geometry in this part.
\subsubsection{Arakelov degree}
Let $\E$ be a Hermitian vector bundle over $\spec\O_K$, and $\{s_1,\ldots,s_r\}$ be a $K$-basis of the vector space $E_K$. The \textit{Arakelov degree} of $\E$ is defined as
\begin{eqnarray*}
  \adeg(\E)&=&-\sum_{v\in M_{K}}[K_v:\Q_v]\log\left\|s_1\wedge\cdots\wedge s_r\right\|_v\\
  &=&\log\left(\#\left(E/\O_Ks_1+\cdots+\O_Ks_r\right)\right)-\frac{1}{2}\sum_{v\in M_{K,\infty}}\log\det\left(\langle s_i,s_j\rangle_{v,1\leqslant i,j\leqslant r}\right),
\end{eqnarray*}
where $\left\|s_1\wedge\cdots\wedge s_r\right\|_v$ follows the definition in \cite[2.1.9]{Chen10b} for all $v\in M_{K,\infty}$, and $\langle s_i,s_j\rangle_{v,1\leqslant i,j\leqslant r}$ is the Gram matrix of the basis $\{s_1,\ldots,s_r\}$ with respect to $v\in M_{K,\infty}$. For those $v\in M_{K,f}$, we take the norms given by models.

We refer the readers to \cite[2.4.1]{Gillet-Soule91} for a proof of the equivalence of the above two definitions. The Arakelov degree is independent of the choice of the basis $\{s_1,\ldots,s_r\}$ by the product formula \eqref{product formula}. In addition, we define
\[\adeg_n(\E)=\frac{1}{[K:\Q]}\adeg(\E)\]
as the \textit{normalized Arakelov degree} of $\E$, which is independent of the choice of the base field $K$.
\subsubsection{Slope}
Let $\E$ be a non-zero Hermitian vector bundle over $\spec\O_K$, and $\rg(E)$ be the rank of $E$. The \textit{slope} of $\E$ is defined as
\[\wmu(\E):=\frac{1}{\rg(E)}\adeg_n(\E).\]
In addition, we denote by $\wmu_{\max}(\E)$ the maximal slope of all its non-zero Hermitian sub-bundles, and by $\wmu_{\min}(\E)$ the minimal slope of all its non-zero Hermitian quotients bundles of $\E$.
\subsubsection{Height of linear maps}
Let $\E$ and $\F$ be two non-zero Hermitian vector bundles over $\spec\O_K$, and $\phi:\; E_K\rightarrow F_K$ be a non-zero homomorphism of $K$-vector spaces. The \textit{height} of $\phi$ is defined as \[h(\phi)=\frac{1}{[K:\Q]}\sum_{v\in M_K}\log\|\phi\|_v,\]
where $\|\phi\|_v$ is the operator norm of $K_v$-linear map $\phi_v:E\otimes_KK_v\rightarrow F\otimes_KK_v$ induced by the above linear homomorphism with respect to every $v\in M_K$.

We refer the readers to \cite[Appendix A]{BostBour96} for some equalities and inequalities on Arakelov degrees and the heights of corresponding homomorphisms.

\subsection{Arithmetic Hilbert-Samuel function}\label{basic setting}
Let $\overline{\mathcal E}$ be a Hermitian vector bundle of rank $n+1$ over $\spec\O_K$, and $\mathbb P(\sE)$ be the projective space which represents the functor from the category of commutative $\O_K$-algebras to the category of sets mapping all $\O_K$-algebra $A$ to the set of projective quotient $A$-module of $\sE\otimes_{\O_K}A$ of rank $1$.

Let $\O_{\mathbb P (\sE)}(1)$ (or by $\O(1)$ if there is no confusion) be the universal bundle, and $\O_{\mathbb P (\sE)}(D)$ (or by $\O(D)$) be the line bundle $\O_{\mathbb P (\sE)}(1)^{\otimes D}$ for simplicity. The Hermitian metrics on $\sE$ induce by quotient of Hermitian metrics (i.e. Fubini-Study metrics) on $\O_{\mathbb P(\sE)}(1)$ which define a Hermitian line bundle $\overline{\O_{\mathbb P(\sE)}(1)}$ on $\mathbb P(\sE)$.

For every $D\in\mathbb N^+$, let
\begin{equation}\label{definition of E_D}
  E_D=H^0\left(\mathbb P(\sE),\O_{\mathbb P(\sE)}(D)\right),
\end{equation} and $r(n,D)$ be its rank over $\O_K$. In fact, we have
\begin{equation}\label{def of r(n,D)}
  r(n,D)={n+D\choose D}.
\end{equation}

For each $v\in M_{K,\infty}$, we denote by $\|\ndot\|_{v,\sup}$ the norm over $E_{D,v}=E_D\otimes_{\O_K,v}\C$ such that
\begin{equation}\label{definition of sup norm}
  \forall s\in E_{D,v},\;\|s\|_{v,\sup}=\sup_{x\in\mathbb P(\sE_K)_v(\C)}\|s(x)\|_{v,\mathrm{FS}},
\end{equation}
where $\|\ndot\|_{v,\mathrm{FS}}$ is the corresponding Fubini-Study norm.

Next, we will introduce the \textit{metric of John}, see \cite{Thompson96} for a systematic introduction to this notion. In general, for a given symmetric convex body $C$, there exists the unique ellipsoid $J(C)$, called \textit{ellipsoid of John}, contained in $C$ whose volume is maximal.

For the $\O_K$-module $E_D$ and any place $v\in M_{K,\infty}$, we take the ellipsoid of John of its unit closed ball defined via the norm$\|\ndot\|_{v,\sup}$, and this ellipsoid induces a Hermitian norm, noted by $\|\ndot\|_{v,\mathrm{John}}$. For every section $s\in E_{D}$, the inequality
\begin{equation}\label{john norm}
  \|s\|_{v,\sup}\leqslant\|s\|_{v,\mathrm{John}}\leqslant\sqrt{r(n,D)}\|s\|_{v,\sup}
\end{equation}
is verified by \cite[Theorem 3.3.6]{Thompson96}. In fact, these constants do not depend on the choice of the symmetric convex body.

Let $A$ be a ring, and $E$ be an $A$-module. We denote by $\sym^D_{A}(E)$ the symmetric product of degree $D$ of the $A$-module $E$, or by $\sym^D(E)$ if there is no confusion on the base ring.

If we consider the above $E_D$ defined in \eqref{definition of E_D} as an $\O_K$-module, we have the isomorphism of $\O_K$-modules $E_D\cong \sym^D(\mathcal{E})$. Then for every place $v\in M_{K,\infty}$, the Hermitian norm $\|\ndot\|_v$ over $\mathcal{E}_{v,\C}$ induces a Hermitian norm $\|\ndot\|_{v,\mathrm{sym}}$ over $E_D$ by the symmetric product. More precisely, this norm is the quotient norm induced by the quotient morphism
\[\sE^{\otimes D}\rightarrow\sym^D(\sE),\]
 where the vector bundle $\overline{\sE}^{\otimes D}$ is equipped with the norms induced by the tensor product of $\overline{\sE}$ over $\spec\O_K$ (see \cite[D\'efinition 2.10]{Gaudron08} for the definition). We say that this norm is the \textit{symmetric norm} over $\sym^D(\sE)$. For any place $v\in M_{K,\infty}$, the norms $\|\ndot\|_{v,\mathrm{John}}$ and $\|\ndot\|_{v,\mathrm{sym}}$ are invariant under the action of the unitary group $U(\sE_{v,\C},\|\ndot\|_v)$ of order $n+1$. Then they are proportional and the ratio is independent of the choice of $v\in M_{K,\infty}$ (see \cite[Lemma 4.3.6]{BGS94} for a proof). We denote by $R_0(n,D)$ the constant such that, for every section $0\neq s\in E_{D,v}$, the equality
\begin{equation}\label{symmetric norm vs John norm}
  \log\|s\|_{v,\mathrm{John}}=\log\|s\|_{v,\mathrm{sym}}+R_0(n,D).
\end{equation}
is verified.
\begin{defi}\label{definition of E_D with norm}
Let $E_D$ be the $\O_K$-module defined in \eqref{definition of E_D}. For every place $v\in M_{K,\infty}$, we denote by $\E_D$ the Hermitian vector bundle over $\spec\O_K$, where for every $v\in M_{K,\infty}$, $E_D$ is equipped with the norm of John $\|\ndot\|_{v,\mathrm{John}}$ induced by the norm $\|\ndot\|_{v,\sup}$ defined in \eqref{definition of sup norm}. Similarly, we denote by $\E_{D,\mathrm{sym}}$ the Hermitian vector bundle over $\spec\O_K$ where $E_D$ is equipped with the norms $\|\ndot\|_{v,\mathrm{sym}}$ introduced above.
\end{defi}
With all the notations in Definition \ref{definition of E_D with norm}, we have the following result.
\begin{prop}[\cite{Chen1}, Proposition 2.7]\label{symmetric norm vs John norm, constant}
  With all the notations in Definition \ref{definition of E_D with norm}, we have
\[\wmu_{\min}(\E_D)=\wmu_{\min}(\E_{D,\mathrm{sym}})-R_0(n,D).\]
In the above equality, the constant $R_0(n,D)$ defined in the equality \eqref{symmetric norm vs John norm} satisfies the inequality
\begin{equation*}
  0\leqslant R_0(n,D)\leqslant\log\sqrt{r(n,D)},
\end{equation*}
where the constant $r(n,D)=\rg(E_D)$ follows the definition in the equality \eqref{def of r(n,D)}.
\end{prop}

Let $X$ be a pure dimensional closed subscheme of $\mathbb{P}(\mathcal{E}_K)$, and $\mathscr{X}$ be the Zariski closure of $X$ in $\mathbb{P}(\mathcal{E})$. We denote by
\begin{equation}\label{evaluation map}
\eta_{X,D}:\;E_{D,K}=H^0\left(\mathbb{P}(\mathcal{E}_K),\O_{\mathbb P(\sE_K)}(D)\right)\rightarrow H^0\left(X,\O_{\mathbb P(\sE_K)}(1)|_X^{\otimes D}\right)
\end{equation}
the \textit{evaluation map} over $X$ induced by the closed immersion from $X$ to $\mathbb P(\sE_K)$. In addition, we denote by $F_D$ the largest saturated sub-$\O_K$-module of $H^0\left(\mathscr{X},\O_{\mathbb P(\sE)}(1)|_\mathscr{X}^{\otimes D}\right)$ such that $F_{D,K}=\im(\eta_{X,D})$. When the integer $D$ is large enough, the homomorphism $\eta_{X,D}$ is surjective, which means $F_D=H^0(\mathscr{X},\O_{\mathbb P(\sE)}(1)|_\mathscr{X}^{\otimes D})$ (cf. \cite[Chap. III, Theomrem 5.2 (b)]{GTM52}).

The $\O_K$-module $F_D$ is equipped with the quotient metrics (from $\E_D$) such that $F_D$ is a Hermitian vector bundle over $\spec \O_K$, noted by $\F_D$ this Hermitian vector bundle. Moreover, in the remainder part of this article, we denote by $r_1(D)$ the rank of the $\O_K$-module $F_D$.
\begin{defi}\label{arithmetic hilbert function}
We denote by $\F_D$ the Hermitian vector bundle over $\spec\O_K$ defined above from \eqref{evaluation map}. We define that the function which maps the positive integer $D$ to $\wmu(\F_D)$ is the \textit{arithmetic Hilbert-Samuel function} of $X$ with respect to the Hermitian line bundle $\overline{\O_{\mathbb P(\sE)}(1)}$.
\end{defi}
\begin{rema}\label{definition of arakelov height}
  With all the notations in Definition \ref{arithmetic hilbert function}. Let
  \begin{equation}\label{definition of arakelov height}
  h_{\overline{\O_{\mathbb P(\sE)}(1)}}(X)=\adeg_n\left(\widehat{c}_1\left(\overline{\O_{\mathbb P(\sE)}(1)}\right)^{d+1}\cdot\left[\mathscr X\right]\right).
  \end{equation}
  In fact, the Arakelov degree \eqref{definition of arakelov height} defines a height of $X$ by the arithmetic intersection theory (cf. \cite[Definition 2.5]{Faltings91}). By \cite[Th\'eor\`eme A]{Randriam06}, we have
  \[h_{\overline{\O_{\mathbb P(\sE)}(1)}}(X)=\lim_{D\rightarrow+\infty}\frac{\adeg_n(\F_D)}{D^{d+1}/(d+1)!}.\]
\end{rema}
By \cite[Corollary 2.9]{Chen1}, we have the trivial lower bound of $\wmu(\F_D)$
\begin{equation}\label{trivial lower bound of F_D}
  \wmu(\F_D)\geqslant-\frac{1}{2}D\log(n+1).
\end{equation}
\subsection{Height of rational points given by Arakelov theory}
We will give a definition of the height of rational points by Arakelov theory in this part. Let $\overline {\mathcal E}$ be a Hermitian vector bundle of rank $n+1$ over $\spec\O_K$, $P\in \mathbb P(\mathcal E_K)(K)$, and $\mathcal P\in\mathbb P(\mathcal E)(\O_K)$ be its Zariski closure in $\mathbb P(\mathcal E)$. Let $\overline {\O_{\mathbb P(\mathcal E)}(1)}$ be the universal bundle equipped with the corresponding Fubini-Study metric at each $v\in M_{K,\infty}$, then $\mathcal P^*\overline {\O_{\mathbb P(\mathcal E)}(1)}$ is a Hermitian line bundle over $\spec\O_K$. We define the \textit{height} of the rational point $P$ as
\begin{equation}\label{arakelov height}
  h_{\overline {\O_{\mathbb P(\mathcal E)}(1)}}(P)=\adeg_n\left(\mathcal P^*\overline {\O_{\mathbb P(\mathcal E)}(1)}\right).
\end{equation}

In fact, \eqref{arakelov height} is the same as the definition \eqref{definition of arakelov height} when we choose $X$ to be a rational point in $\mathbb P(\sE_K)$ considered as one of its closed integral subschemes.
\begin{rema}\label{Comparing heights}
We keep all the above notations in this part. Now we choose $\overline{\sE}=\left(\O_K^{\oplus(n+1)},\left(\|\ndot\|_v\right)_{v\in M_{K,\infty}}\right)$, where for every $v\in M_{K,\infty}$, $\|\ndot\|_v$ is the $\ell^2$-norm mapping $(t_0,\ldots,t_n)$ to $\sqrt{|v(t_0)|^2+\cdots+|v(t_n)|^2}$. We suppose that $P$ has the $K$-rational projective coordinate $[x_0:\cdots:x_n]$, then we have (cf. \cite[Proposition 9.10]{Moriwaki-book})
\begin{eqnarray*}
  h_{\overline {\O_{\mathbb P(\mathcal E)}(1)}}(P)&=&\sum\limits_{v\in M_{K,f}}\frac{[K_v:\Q_v]}{[K:\Q]}\log \left(\max\limits_{1\leqslant i\leqslant n}|x_i|_v\right)\\
  & &\;\;+\frac{1}{2}\sum\limits_{v\in M_{K,\infty}}\frac{[K_v:\Q_v]}{[K:\Q]}\log\left(\sum\limits_{j=0}^n|v(x_j)|^2\right).
\end{eqnarray*}
In addition, let the $h(\ndot)$ be the height defined in \eqref{log height}. Then by some elementary calculation, the inequality
  \[\left|h(P)-h_{\overline {\O_{\mathbb P(\mathcal E)}(1)}}(P)\right|\leqslant\frac{1}{2}\log(n+1)\]
  is verified uniformly for all $P\in\mathbb P(\sE_K)$ when we choose the above $\overline{\sE}$.
\end{rema}
\subsection{Further notations on counting rational points problem}
Let $\psi:X\hookrightarrow\mathbb P(\sE_K)$ be a closed immersion from $X$ to $\mathbb P(\sE_K)$, and $P\in X(K)$. We denote the height of $P$ by $h_{\overline {\O_{\mathbb P(\mathcal E)}(1)}}(\psi(P))$ at \eqref{arakelov height}. We will use the notations $h_{\overline {\O_{\mathbb P(\mathcal E)}(1)}}(P)$, $h_{\overline {\O(1)}}(P)$ or $h(P)$ if there is no confusion of the morphism $\psi$ and the Hermitian line bundle $\overline{\O_{\mathbb P(\mathcal E)}(1)}$. This height also satisfies the Northcott's property for arbitrary Hermitian vector bundle $\overline{\sE}$ (cf. \cite[Theorem 5.3]{Yuan_ICCM2010}), so it can be used in the counting rational points problem. Actually, the line bundle $\O_{\mathbb P(\sE_K)}(1)$ can be replaced by arbitrary ample line bundle for the correctness of the Northcott's property.

In the rest part of this article, unless specially mentioning, we will use the height function defined at \eqref{arakelov height}, and we will use the notation $h(\ndot)$ to denote this height function. The classic height defined at \eqref{classic absolute height} and \eqref{log height} will not be essentially used any longer.
\section{An improved estimate of the determinant}
In this section, we will improve an estimate in the determinant method. Parts of the construction are from \cite{Salberger2015}.
\subsection{Estimates of norms}\label{estimates of norms}
In this part, we will estimate the norms of some local homomorphisms, which can be viewed as a generalization of parts of \cite[\S 3]{Chen2}. The same idea has been applied in \cite[\S 16.2]{Salberger2015}. This estimate is finer than that in \cite[Lemma 2.4]{Salberger07} and \cite[Proposition 3.4]{Chen2}, but will be more implicit because of some technical obstructions.

First, we refer a useful auxiliary result in \cite{Chen2}, which will be useful in the approach of Arakelov geometry. Before introducing it, we recall an useful notion. Let $(k,|\ndot|)$ be a non-Archimedean field, and $(V,\|\ndot\|)$ be normed vector space over $(k,|\ndot|)$. We say that $(V,\|\ndot\|)$ is \textit{ultranormed} if for all $x,y\in U$, we have $\|x+y\|\leqslant \max\left\{\|x\|,\|y\|\right\}$.
\begin{lemm}[\cite{Chen2}, Lemma 3.3]\label{untrametric operator norm}
  Let $k$ be a field equipped with a non-archimedean absolute value $|\ndot|$, $U$ and $V$ be two $k$-linear ultranormed spaces of finite rank and $\phi:U\rightarrow V$ be a $k$-linear homomorphism. Let $m=\dim_k(U)$. For any integer $1\leqslant i\leqslant m$, let
  \[\lambda_i=\inf_{\begin{subarray}{c}W\subset U\\ \codim_U(W)=i-1\end{subarray}}\|\phi|_W\|.\]
  If $i>m$, let $\lambda_i=0$. Then for any integer $r>0$, we have
  \begin{equation}
    \left\|\wedge^r\phi\right\|\leqslant\prod_{i=1}^r\lambda_i.
  \end{equation}
\end{lemm}

In the remained part of this section, unless specially mentioned, we denote by $K$ a number field, and by $\O_K$ its ring of integers. We fix a Hermitian vector bundle $\overline{\sE}$ of rank $n+1$ over $\spec\O_K$, a closed integral subscheme $X$ of $\mathbb P(\sE_K)$, and the Zariski closure $\mathscr X$ of $X$ in $\mathbb P(\sE)$. We refer the readers to \cite[Lemma 16.9]{Salberger2015} for the original ideas of the construction.

 Let $\p$ be a maximal ideal of $\O_K$, $\f_\p$ be the residue field of $\O_K$ at $\p$. Let $\xi$ be an $\f_\p$-point of $\mathscr X$, and $k\in\mathbb N^+$. We suppose that $\{f_i\}_{1\leqslant i\leqslant k}$ is a family of local homomorphisms of $\O_{K,\p}$-algebras from $\O_{\mathscr X,\xi}$ to $\O_{K,\p}$. Let $\mathfrak a$ be the kernel of $f_1$, then we have $\O_{\mathscr X,\xi}/\mathfrak a\cong\O_{K,\p}$, which shows that $\mathfrak a$ is a prime ideal. Furthermore, since $\O_{\mathscr X,\xi}$ is a local ring with the maximal ideal $\sm_\xi$, we have $\sm_\xi\supseteq\mathfrak a$. Moreover, for $f_1$ is a local homomorphism, we have $\mathfrak a+\p\O_{\mathscr X,\xi}=\sm_\xi$.

In addition, we suppose that the point $\xi$ is regular in $\mathscr X$, which means $\O_{\mathscr X,\xi}$ is a regular local ring. In this case, the ideal $\mathfrak a$ is generated by $\dim\left(\O_{\mathscr X,\xi}\right)-1$ regular parameters (cf. \cite[Proposition 4.10]{LNM146}). Since these elements form a regular sequence on $\O_{\mathscr X,\xi}$ (cf. \cite[Chap. III, Proposition 6]{SerreLocAlg}), we have $\sym^m(\mathfrak a/\mathfrak a^2)\cong\mathfrak a^m/\mathfrak a^{m+1}$ as free $\O_{K,\p}$-modules for all $m\geqslant0$ by \cite[Chap. IV, \S2, Corollary 2.4]{FultonLang1985}, where we define $\mathfrak a^0=\O_{\mathscr X,\xi}$ for convenience.

Let $S=\O_{\mathscr X,\xi}\smallsetminus\mathfrak a$, and we denote by
\begin{equation}\label{R_X,xi}
R_{\mathscr X,\xi}=S^{-1}\left(\O_{\mathscr X,\xi}\right)\end{equation}
 the localization of $\O_{\mathscr X,\xi}$ at the prime ideal $\mathfrak a$. We denote by $m_\xi$ the maximal ideal of the ring $R_{\mathscr X,\xi}$, and then we have $m_\xi=\mathfrak a R_{\mathscr X,\xi}$ by the definition of this localization.

 Let $u\in S$ and $r\in\mathfrak a^m$ for every $m\geqslant0$. If $ur\in \mathfrak a^{m+1}$ is verified, since we have $(u+\mathfrak a)(r+\mathfrak a^{m+1})=\mathfrak a^{m+1}$, then we obtain $r\in \mathfrak a^{m+1}$. Therefore, we deduce
 \begin{equation}\label{m cap a}
   m_\xi^{m+1}\cap\mathfrak a^m=\left(\mathfrak a^{m+1}\cdot R_{\mathscr X,\xi}\right)\cap\mathfrak a^m=\mathfrak a^{m+1}
 \end{equation}
 for all $m\geqslant0$.

Let $E$ be a free sub-$\O_{K,\p}$-module of finite type of $\O_{\mathscr X,\xi}$ and let
\begin{equation}\label{f_i}
  f=(f_i|_E)_{1\leqslant i\leqslant k}:\;E\rightarrow\O_{K,\p}^k
\end{equation}
be an $\O_{K,\p}$-linear homomorphism. As $f_1$ is a homomorphism of $\O_{K,\p}$-algebras, it is surjective.

We consider $\left(E\cap\mathfrak a^j\right)/\left(E\cap \mathfrak a^{j+1}\right)$ and $\left(E\cap m_\xi^j\right)/\left(E\cap m_\xi^{j+1}\right)$ as two free $\O_{K,\p}$-modules, where we consider $E$ as a sub-$\O_{K,\p}$-module of $R_{\mathscr X,\xi}$ if it is necessary. Then we have the isomorphism of $\O_{K,\p}$-modules
\begin{eqnarray}\label{isom of E cap m}
  & &\left(E\cap\mathfrak a^j\right)/\left(E\cap \mathfrak a^{j+1}\right)\cong\left(E\cap\mathfrak a^j\right)/\left((E\cap \mathfrak a^{j})\cap (E\cap m_\xi^{j+1})\right)\\
  &\cong&\left((E\cap\mathfrak a^j)+(E\cap m_\xi^{j+1})\right)/\left(E\cap m_\xi^{j+1}\right)\cong\left(E\cap m_\xi^j\right)/\left(E\cap m_\xi^{j+1}\right)\nonumber
\end{eqnarray}
by \eqref{m cap a}, where we use the fact $\mathfrak a^jR_{\mathscr X,\xi}+m_\xi^{j+1}=m_\xi^j$ in $R_{\mathscr X,\xi}$.
%By the above argument, we have that
%\begin{eqnarray*}
 % \f_\p\otimes_{\O_{K,\p}}\left((E\cap\mathfrak a^j)/(E\cap\mathfrak a^{j+1})\right)&\cong&\left((E\cap\mathfrak a^j)/\p\O_{\mathscr X,\xi}\right)/\left((E\cap\mathfrak a^{j+1})/\p\O_{\mathscr X,\xi}\right)\\
 % &\cong&\left((E\cap\mathfrak a^j)/\p\O_{\mathscr X,\xi}\right)/\left((E\cap\sm_\xi^{j+1})/\p\O_{\mathscr X,\xi}\right)\\
  %&\cong&\left((E\cap\sm_\xi^j)/\p\O_{\mathscr X,\xi}\right)/\left((E\cap\sm_\xi^{j+1})/\p\O_{\mathscr X,\xi}\right),
%\end{eqnarray*}
%where the second isomorphism above comes from the isomorphism theorems of vector spaces, and the third one is due to the fact $\mathfrak a+\p\O_{\mathscr X,\xi}=\sm_\xi$ mentioned above. By Nakayama's lemma (cf. \cite[Theorem 2.2]{Matsumura}), the rank of $\left(E\cap\mathfrak a^j\right)/\left(E\cap\mathfrak a^{j+1}\right)$ over $\O_{K,\p}$ is equal to the rank of $\left(E\cap\sm_\xi^j\right)/\left(E\cap\sm_\xi^{j+1}\right)$ over $\O_{K,\p}$.

Now we suppose that the reductions of all the above local homomorphisms $f_1,\ldots,f_k$ modulo $\p$ are same, which means all the composed homomorphisms $\O_{\mathscr X,\xi}\xrightarrow{f_i} \O_{K,\p}\rightarrow\f_\p$ are same for every $i=1,\ldots,k$, where the last arrow is the canonical reduction morphism modulo $\p$. Let $N(\p)=\#\f_\p$. In this case, the norm of the restriction of $f$ on $E\cap \mathfrak a^j$ is smaller than $N(\p)^{-j}$. In fact, for any $1\leqslant i\leqslant k$, we have $f_i(\mathfrak a)\subset\p\O_{K,\p}$, and hence we have $f_i(\mathfrak a^j)\subset\p^j\O_{K,\p}$.

From the above construction, we have the following result, which is a reformulation of the estimate in \cite[Lemma 16.9]{Salberger2015}.
\begin{prop}\label{upper bound of operator norm}
  Let $\p$ be a maximal ideal of $\O_K$, and $\xi\in\mathscr X(\f_\p)$ be a non-singular point. Suppose that $\{f_i\}_{1\leqslant i\leqslant k}$ is a family of local $\O_{K,\p}$-linear homomorphisms from $\O_{\mathscr X,\xi}$ to $\O_{K,\p}$ whose reductions module $\p$ are same. Let $E$ be a free sub-$\O_{K,\p}$-module of finite type of $\O_{\mathscr X,\xi}$, $f=(f_i|_E)_{1\leqslant i\leqslant k}$ be that defined in \eqref{f_i}, and $N(\p)=\#(\O_K/\p)$. We consider $E$ as a sub-$\O_{K,\p}$-module of $R_{\mathscr X,\xi}$, and let
  \begin{equation}\label{R(E)}
    \mathcal R_\xi(E)=\sum_{k=1}^\infty\dim_{K}\left(E\cap m_\xi^k\right)_{K}.
  \end{equation}
  Then if $r=\dim_K(E_K)$, we have
  \[\log\|\wedge^rf_K\|\leqslant -\mathcal R_\xi(E)\log N(\p).\]
\end{prop}
\begin{proof}
By the above notations and argument, we have the filtration
\[\mathcal F:\; E\supset E\cap \mathfrak a\supset\cdots\supset E\cap \mathfrak a^j\supset E\cap \mathfrak a^{j+1}\supset\cdots\]
of $E$, whose $j$-th subquotient $\left(E\cap\mathfrak a^j\right)/\left(E\cap \mathfrak a^{j+1}\right)$ is a free $\O_{K,\p}$-module. The restriction of $f$ on $E\cap \mathfrak a^j$ has norm smaller than $N(\p)^{-j}$. Meanwhile, let $\{q_\xi(m)\}_{m=1}^\infty$ be the series of non-negative integers where the integer $m$ appears exactly
\[\dim_K\left(E\cap m_\xi^m\right)_K-\dim_K\left(E\cap m_\xi^{m+1}\right)_K\]
 times. Then by the isomorphism \eqref{isom of E cap m}, the free $\O_{K,\p}$-modules $\left(E\cap\mathfrak a^j\right)/\left(E\cap\mathfrak a^{j+1}\right)$ and $\left(E\cap m_\xi^j\right)/\left(E\cap m_\xi^{j+1}\right)$ have the same rank for all $j\geqslant0$. Thus we have
 \begin{equation}\label{q without wedge}
   \inf_{\begin{subarray}{c}W\subset E_K\\ \codim_{E_K}(W)=j-1\end{subarray}}\|f_K|_W\|\leqslant N(\p)^{-q_\xi(j)}.
 \end{equation}
 Since the above filtration $\mathcal F$ is of finite length, then by an elementary calculation, we obtain the equality
\[\sum\limits_{m=1}^\infty q_\xi(m)=\sum\limits_{m=1}^\infty\dim_K\left(E\cap m_\xi^m\right)_K.\]
  Finally by applying Lemma \ref{untrametric operator norm} to \eqref{q without wedge}, we obtain the result.
\end{proof}
\subsection{Existence of auxiliary hypersurfaces}
In this part, we will reformulate the determinant method by the slope method. Different from \cite[Theorem 3.2]{Salberger07} and \cite[Theorem 16.12]{Salberger2015}, our estimate will depend on the term $\mathcal R_\xi(E)$ defined in \eqref{R(E)} for a special choice of $E$. In \S4, we will reformulate the estimate of $\mathcal R_\xi(E)$ for our application such that we are able to control the number of auxiliary hypersurfaces by this result. The strategy is similar to that of \cite[Theorem 3.1]{Chen2}.

The following slope equality is useful in this reformulation, which is obtained by the slope equalities and inequalities.
\begin{prop}[\cite{Chen1}, Proposition 2.2]\label{slope of evaluation map}
  Let $\overline E$ be a Hermitian vector bundle of rank $r>0$ over $\spec\O_K$, and $\{\overline L_i\}_{i\in I}$ be a family of Hermitian line bundles over $\spec\O_K$. If $\phi:\; E_K\rightarrow\bigoplus\limits_{i\in I}L_{i,K}$ is an injective homomorphism of $K$-vector spaces, then there exists a subset $I_0$ of $I$ whose cardinality is $r$ such that the equality
\[\wmu(\E)=\frac{1}{r}\left(\sum_{i\in I_0}\wmu(\overline L_i)+h\left(\wedge^r(\pr_{I_0}\circ\phi)\right)\right)\]
is verified, where $\pr_{I_0}:\;\bigoplus\limits_{i\in I}L_{i,K}\rightarrow\bigoplus\limits_{i\in I_0}L_{i,K}$ is the canonical projection.
\end{prop}
\
The following result is a refined determinant method, which follows the strategy of \cite[Theorem 3.1]{Chen2} by bringing the term $\mathcal R_\xi(E)$ defined in \eqref{R(E)} into the estimate.

Before providing the statement, we will introduce the operation below. Let $\overline{\sE}$ be a Hermitian vector bundle of rank $n+1$ over $\spec\O_K$, $X$ be a closed integral subscheme of $\mathbb P(\sE_K)$, and $\mathscr X$ be the Zariski closure of $X$ in $\mathbb P(\sE)$. We choose a $P\in X(K)$, and let $\mathcal P\in\mathscr X(\O_K)$ be the Zariski closure of $P$ in $\mathscr X$. If we say that the reduction of $P$ modulo a maximal ideal $\p$ of $\O_K$ is $\xi\in\mathscr X(\f_\p)$, we mean that we consider the reduction of $\mathcal P$ modulo $\p$, whose image is $\xi$. We will use this representation multiple times in this article below.
\begin{theo}\label{semi-global determinant method}
  We keep all the above notations. Let $\{\p_j\}_{j\in J}$ be a finite family of maximal ideals of $\O_K$, and $\{P_i\}_{i\in I}$ be a family of rational points of $X$ such that, for any $i\in I$ and any $j\in J$, the reduction of $P_i$ modulo $\p_j$ coincides with the same non-singular point $\xi_j\in\mathscr X(\f_{\p_j})$. Let $\F_D$ be that defined in Definition \ref{arithmetic hilbert function}, $\mathcal R_{\xi_j}(F_D)$ be that defined in \eqref{R(E)}, $r_1(D)=\rg(F_D)$, $N(\p_j)=\#(\O_K/\p_j)$, and the height function $h(\ndot)$ of rational points defined in \eqref{arakelov height} by Arakelov theory. If the inequality
    \begin{equation}
   \sup_{i\in I}h(P_i)<\frac{\wmu(\overline{F}_D)}{D}-\frac{\log r_1(D)}{2D}+\frac{1}{[K:\Q]}\sum_{j\in J}\frac{\mathcal R_{\xi_j}(F_{D})}{Dr_1(D)}\log N(\p_j)
  \end{equation}
  is verified for a positive integer $D$, then there exists a section $s\in E_{D,K}$ (see \eqref{definition of E_D} for its definition), which contains $\{P_i\}_{i\in I}$ but does not contain the generic point of $X$. In other words, $\{P_i\}_{i\in I}$ can be covered by a hypersurfaces of $\mathbb P(\sE_K)$ of degree $D$ which does not contain the generic point of $X$.
\end{theo}
\begin{proof}
We suppose the section predicted by this theorem does not exist. Then the evaluation map
\[f:\;F_{D,K}\rightarrow \bigoplus\limits_{i\in I}P_i^*\O_{\mathbb P(\sE_K)}(1)|_X^{\otimes D}\]
is injective. We can replace $I$ by one of its subsets such that the above homomorphism $f$ is an isomorphism.

For every $v\in M_{K,\infty}$, we have
\[\frac{1}{r_1(D)}\log\|\wedge^{r_1(D)}f\|_v\leqslant\log\|f\|_v\leqslant\log\sqrt{r_1(D)},\]
where the first inequality comes from Hadamard's inequality, and the second one is due to the definition of metrics of John introduced at \S\ref{basic setting}.

For every $v\in M_{K,f}$, let $\p$ be the maximal ideal of $\O_K$ corresponding to the place $v$. By definition, the isomorphism $f$ is induced by a homomorphism $\O_K$-modules
\[F_D\rightarrow\bigoplus_{i\in I}\mathcal P_i^*\O_{\mathbb P(\sE)}(1)|_{\mathscr X}^{\otimes D},\]
where $\mathcal P_i$ is the $\O_K$-point of $\mathscr X$ extending $P_i$. Hence for any maximal ideal $\p$, we have $\log\|\wedge^{r_1(D)}f\|_\p\leqslant0$.

We fix a $j\in J$. For each $i\in I$, the $\O_K$-point $\mathcal P_i$ defines a local homomorphism from $\O_{\mathscr X,\xi_j}$ to $\O_{K,\p_j}$ which is $\O_{K,\p_j}$-linear. By taking a local trivialization of $\O_{\mathbb P(\sE)}(1)|_{\mathscr X}^{\otimes D}$ at $\xi_j$, we identify $F_D$ as a sub-$\O_{K,\p_j}$-module of $\O_{\mathscr X,\xi_j}$. Then by Proposition \ref{upper bound of operator norm}, we have
\[\log\|\wedge^{r_1(D)}f\|_{\p_j}\leqslant-\mathcal R_{\xi_j}(F_D)\log N(\p_j).\]

From the above two upper bounds of the operator norms, combined with Proposition \ref{slope of evaluation map}, we obtain
\[\frac{\wmu(\F_D)}{D}\leqslant\sup_{i\in I}h(P_i)+\frac{1}{2D}\log r_1(D)-\frac{1}{[K:\Q]}\sum_{j\in J}\frac{\mathcal R_{\xi_j}(F_D)}{Dr_1(D)}\log N(\p_j),\]
which leads to a contradiction.
\end{proof}
\section{Estimates of $\mathcal R_{\xi_j}(F_D)$}
In order to apply Theorem \ref{semi-global determinant method}, more information about the term $\mathcal R_{\xi_j}(F_D)$ need to be gathered. The aim of this section is to give an asymptotic estimate of $\mathcal R_{\xi_j}(F_D)$, which reformulate a result of Salberger by a more implicit approach.
\subsection{Finiteness of $\mathcal R_{\xi_j}(F_D)$}
Formally, the sum in $\mathcal R_{\xi_j}(F_D)$ defined in \eqref{R(E)} is infinite. But since the filtration $\mathcal F$ introduced in the proof of Proposition \ref{upper bound of operator norm} is finite, then $\mathcal R_{\xi_j}(F_D)$ is essentially a finite sum. Then when the positive integer $m$ is large enough in $F_D\cap m_{\xi_j}^m$, it will be a zero module, so essentially it is a finite sum.

The following result is a reformulation of \cite[Lemma 16.10]{Salberger2015}, at which the case of cubic hypersurfaces in $\mathbb P^3$ was considered only.
\begin{prop}\label{R(E)-kernel}
  We keep all notations and conditions in Theorem \ref{semi-global determinant method}. Let $\eta_j\in X(K)$ be a rational point which specializes to $\xi_j$ with respect to the operation in Theorem \ref{semi-global determinant method}, $m_{\xi_j}$ be the maximal ideal of $R_{\mathscr X,\xi_j}$ defined in \eqref{R_X,xi}, and $\sn_{\eta_j}$ be the maximal ideal of $\O_{X}$ at the point $\eta_j$. Then for every $m\in\mathbb N^+$ and $j\in J$ in Theorem \ref{semi-global determinant method}, we have
  \begin{eqnarray*}
    \dim_K\left(F_{D}\cap m_{\xi_j}^m\right)_K&=&\dim_K\ker\left(F_{D,K}\rightarrow H^0\left(X,\O_{\mathbb P(\sE_K)}(1)|_X^{\otimes D}\otimes\O_{X}/\sn_{\eta_j}^m\right)\right)\\
    &\geqslant&\max\left\{0,r_1(D)-{d+m-1\choose m-1}\right\},
  \end{eqnarray*}
  where we identify $F_D$ as a sub-$\O_{K,\p_j}$-module of $\O_{\mathscr X,\xi_j}$ for the above $j\in J$.
\end{prop}

\begin{proof}
  Let $s_1,\ldots,s_{r_1(D)}\in F_D$ which generate $F_D$. Let $T_0,\ldots,T_n$ be the homogeneous coordinate of $\mathscr X\hookrightarrow\mathbb P(\sE)$. Without loss of generality, we suppose that $T_0(\xi_j)\neq0$ with respect to the canonical morphism. Let $r_i=s_i/T_0^D$ for all $i=1,\ldots,r_1(D)$. and $W_D\subset R_{\mathscr X,\xi_j}$ be the vector space over $K$ generated by the images of $r_1,\ldots,r_{r_1(D)}$ in $R_{\mathscr X,\xi_j}$, which is also of dimension $r_1(D)$. Thus for each $s\in F_D$, its image in $H^0\left(X,\O_{\mathbb P(\sE_K)}(1)|_X^{\otimes D}\otimes\O_{X}/\sn_{\eta_j}^m\right)$ is zero if and only if $s/T_0^D\in\ker\left(W_D\rightarrow W_D/m_{\xi_j}^m\right)$ considered as an element in $R_{\mathscr X,\xi_j}$, which means it is verified if and only if $s/T_0^D\in W_D\cap m_{\xi_j}^m$. Thus there exists an isomorphism of $K$-vector spaces from $F_{D,K}$ to $W_D$, which maps $\ker\left(F_{D,K}\rightarrow H^0\left(X,\O_{\mathbb P(\sE_K)}(1)|_X^{\otimes D}\otimes\O_{X}/\sn_{\eta_j}^m\right)\right)$ onto $W_D\cap m_{\xi_j}^m$, and then we obtain the first equality in the assertion.

  By the fact that the point $\xi_j$ is regular in $\mathscr X$ and $\dim(X)=d$, then the point $\eta_j$ is also regular in $X$, and the ring $R_{\mathscr X,\xi_j}$ is a regular local ring of Krull dimension $d$. By these facts, we have $\dim_{K}\left(R_{\mathscr X,\xi_j}/m_{\xi_j}^m\right)={d+m-1\choose m-1}$ for all $m\in\mathbb N^+$. Furthermore, we have $\dim_K\left(W_D/\left(W_D\cap m_{\xi_j}^m\right)\right)\leqslant\dim_K\left(R_{\mathscr X,\eta_j}/m_{\xi_j}^m\right)$. Hence we have
  \begin{eqnarray*}
    \dim_K\left(F_D\cap m_{\xi_j}^m\right)_K&=&\dim_K\left(W_D\right)-\dim_K\left(W_D/\left(W_D\cap m_{\xi_j}^m\right)\right)\\
    &\geqslant& r_1(D)-{d+m-1\choose m-1},
  \end{eqnarray*}
  which completes the proof.
\end{proof}

\subsubsection*{Connection with Seshadri constant}
 In this part, we will give a lower bound of the positive integer $m$ such that
 \[\dim_K\left(F_{D}\cap m_{\xi_j}^m\right)_K=\dim_K\ker\left(F_{D,K}\rightarrow H^0\left(X,\O_{\mathbb P(\sE_K)}(1)|_X^{\otimes D}\otimes\O_{X}/\sn_{\eta_j}^m\right)\right)\]
  are both zero, where all the above notations are same as those in Proposition \ref{R(E)-kernel}. For this target, we will introduce some notions on the geometric positivity of line bundles. We refer the readers to \cite[\S5.1]{LazarsfeldI} for a systemic introduction to it.

Let $X$ be an closed integral projective scheme over a field, $L$ be a line bundle on $X$, and $\xi\in X$ be a regular point with the maximal ideal $\sn_\xi\subset \O_X$. We consider the natural map
\begin{equation}\label{separates s-jets}
  H^0\left(X,L\right)\rightarrow H^0\left(X,L\otimes\O_X/\sn_\xi^{s+1}\right)
\end{equation}
taking the global sections of $L$ to their $s$-jets at $\xi$. By definition, the kernel of the map \eqref{separates s-jets} is $H^0\left(X, L\otimes \sn_\xi^{s+1}\right)$.

In addition, let $L$ be a nef line bundle on $X$. We fix a closed point $\xi\in X$, and let $\pi:\;\widetilde X\rightarrow X$ be the blowing up at $\xi$, and $E=\pi^{-1}(\xi)$ be the exceptional divisor. We define the \textit{Seshadri constant} of $L$ at $\xi$ as
\begin{equation}
  \epsilon(X,L;\xi)=\epsilon(L,\xi)=\sup\{\epsilon>0|\;\pi^*L-\epsilon E\hbox{ is nef }\}.
\end{equation}

By \cite[Proposition 5.1.5]{LazarsfeldI}, we have
\begin{equation}\label{seshadri constant by intersection}
  \epsilon(L;\xi)=\inf_{\xi\in C\subseteq X}\left\{\frac{(L\cdot C)}{\mu_\xi(C)}\right\},
\end{equation}
where $C$ takes over all integral curves $C\subseteq X$ passing through $\xi$, and $\mu_\xi(C)$ is the multiplicity of $\xi$ in $C$, see \S \ref{local algebra} for the definition.

Some properties of the Seshadri constant will be useful in the proof of the proposition below.
\begin{prop}\label{bound of finite sum}
  With all the notations and conditions in Proposition \ref{R(E)-kernel}, when $m\geqslant \left[\sqrt[d]{\delta}D\right]+1$, we have
  \[\ker\left(F_{D,K}\rightarrow H^0\left(X,\O_{\mathbb P(\sE_K)}(1)|_X^{\otimes D}\otimes\O_{X}/\sn_{\eta_j}^m\right)\right)=0,\]
  where $\left[s\right]$ denotes the largest integer smaller than $s$.
\end{prop}
\begin{proof}
By the definition of $F_{D,K}$ induced in \eqref{evaluation map}, the $K$-vector space $F_{D,K}$ is a sub-$K$-vector space of $H^0\left(X,\O_{\mathbb P(\sE_K)}(1)|_X^{\otimes D}\right)$, so it is enough to prove the bound for the $K$-linear map
\begin{equation*}
  H^0\left(X,\O_{\mathbb P(\sE_K)}(1)|_X^{\otimes D}\right)\rightarrow H^0\left(X,\O_{\mathbb P(\sE_K)}(1)|_X^{\otimes D}\otimes\O_{X}/\sn_{\eta_j}^m\right).
\end{equation*}
In other words, we need a bound of $m\in\mathbb N$ such that $H^0\left(X, \O_{\mathbb P(\sE_K)}(1)|_X^{\otimes D}\otimes \sn_{\eta_j}^{m}\right)$ is zero.

By definition, the space $H^0\left(X, \O_{\mathbb P(\sE_K)}(1)|_X^{\otimes D}\otimes \sn_{\eta_j}^{m}\right)$ is zero when $m$ is strictly larger than the possibly maximal multiplicity of the point $\eta_j$ in the divisors which are linearly equivalent to $\O_{\mathbb P(\sE_K)}(1)|_X^{\otimes D}$. We denote by $\mu_{\eta_j}\left(\left|\O_{\mathbb P(\sE_K)}(1)|_X^{\otimes D}\right|\right)$ the above maximal multiplicity. By \cite[Corollary 12.4]{Fulton} and \eqref{seshadri constant by intersection}, we have
\begin{equation}\label{seshadri constant>multiplicity}
  \mu_{\eta_j}\left(\left|\O_{\mathbb P(\sE_K)}(1)|_X^{\otimes D}\right|\right)\leqslant\epsilon\left(\O_{\mathbb P(\sE_K)}(1)|_X^{\otimes D},\eta_j\right),
\end{equation}
where we consider the intersection in the regular locus of $X$, and the multiplicity of a point in pure-dimensional schemes is considered at \cite[Corollary 12.4]{Fulton}. In addition, the multiplicity satisfies the additivity of cycles by \cite[Chap. VIII, \S 7, $\mathrm{n}^\circ$ 1, Prop. 3]{Bourbaki83}.

By \cite[Example 5.1.4]{LazarsfeldI}, we have
\begin{equation}\label{homogenesous of seshadri constant}
  \epsilon\left(\O_{\mathbb P(\sE_K)}(1)|_X^{\otimes D};\eta_j\right)=D\epsilon\left(\O_{\mathbb P(\sE_K)}(1)|_X;\eta_j\right).
\end{equation} By \cite[Proposition 5.1.9]{LazarsfeldI}, we have
\begin{equation}\label{bound of seshadri constant}
  \epsilon\left(\O_{\mathbb P(\sE_K)}(1)|_X;\eta_j\right)\leqslant\sqrt[d]{\frac{\O_{\mathbb P(\sE_K)}(1)|_X^d}{\mu_{\eta_j}\left(X\right)}}=\sqrt[d]{\delta},
\end{equation}
for $\eta_j$ is regular in $X$ and $\deg(X)=\delta$ with respect to $\O(1)$.

By \eqref{seshadri constant>multiplicity}, \eqref{homogenesous of seshadri constant} and \eqref{bound of seshadri constant}, when $m\geqslant \left[\sqrt[d]{\delta}D\right]+1$, we have $m>\mu_{\eta_j}\left(|\O_{\mathbb P(\sE_K)}(1)|_X^{\otimes D}|\right)$, and we will have the trivial kernel in this case.
\end{proof}
\subsection{Invariants induced by blowing up}
Let $\overline \sE$ be a Hermitian vector bundle of rank $n+1$ over $\spec\O_K$, $X$ be a closed integral subscheme of $\mathbb P(\sE_K)$ of dimension $d$ and degree $\delta$, and $\mathscr X$ be the Zariski closure of $X$ in $\mathbb P(\sE)$. If the positive integer $D$ is large enough, then we have $F_D=H^0\left(\mathscr X,\O_{\mathbb P(\sE)}(1)|_{\mathscr X}^{\otimes D}\right)$ and $F_{D,K}=H^0\left(X,\O_{\mathbb P(\sE_K)}(1)|_X^{\otimes D}\right)$, where $F_D$ and $F_{D,K}$ are defined in Definition \ref{arithmetic hilbert function}. By this fact, we will give an alternative description of the term $\mathcal R_{\xi_j}(F_D)$ in Theorem \ref{semi-global determinant method}.

Let $\eta\in X(K)$ be non-singular, $\sn_{\eta}$ be the maximal ideal of $\O_X$ at the point $\eta$, and
\begin{equation}\label{blowing up at one point}
  \pi:\widetilde{X}\rightarrow X
\end{equation}
be the blowing up of $X$ at $\eta$. Let $E=\pi^{-1}(\eta)$ be the exceptional divisor of the above blowing up morphism $\pi$, and $I_E\subset\O_{\widetilde{X}}$ be the ideal sheaf of $E\subset \widetilde{X}$. By the projection formula (cf. \cite[Chap. III, Exercise 8.3]{GTM52}) applied at \eqref{blowing up at one point}, we have $R^i\pi_*\left(\pi^*\left(\O_{\mathbb P(\sE_K)}(1)|_X^{\otimes D}\right)\right)=0$ for all $i\geqslant1$, and it deduces $\pi_*\left(\pi^*\left(\O_{\mathbb P(\sE_K)}(1)|_X^{\otimes D}\right)\right)=\O_{\mathbb P(\sE_K)}(1)|_X^{\otimes D}$. So we obtain
\[H^0\left(X,\O_{\mathbb P(\sE_K)}(1)|_X^{\otimes D}\right)\cong H^0\left(\widetilde{X},\pi^*\left(\O_{\mathbb P(\sE_K)}(1)|_X^{\otimes D}\right)\right).\]
From the above isomorphism, we have the commutative diagram
\[\xymatrix{H^0\left(X,\O_{\mathbb P(\sE_K)}(1)|_X^{\otimes D}\right)\ar[d]\ar[r]&H^0\left(X,\O_{\mathbb P(\sE_K)}(1)|_X^{\otimes D}\otimes\O_{X}/\sn_\eta^m\right)\ar[d]\\H^0\left(\widetilde{X},\pi^*\left(\O_{\mathbb P(\sE_K)}(1)|_X^{\otimes D}\right)\right)\ar[r]&H^0\left(\widetilde{X},\pi^*\left(\O_{\mathbb P(\sE_K)}(1)|_X^{\otimes D}\right)\otimes\O_{\widetilde{X}}/I_E^m\right),}\]
where the kernel of the bottom map is isomorphic to $H^0\left(\widetilde{X}, \pi^*\left(\O_{\mathbb P(\sE_K)}(1)|_X^{\otimes D}\right)\otimes I_E^m\right)$ for $m\geqslant1$. By the above argument, we have the following result.
\begin{prop}\label{kernel-blowup}
  With all the above notations, we have
  \begin{eqnarray*}
    & &\dim_K\left(H^0\left(\widetilde{X}, \pi^*(\O_{\mathbb P(\sE_K)}(1)|_X^{\otimes D})\otimes I_E^m\right)\right)\\
    &=&\dim_K\ker\left(H^0\left(X,\O_{\mathbb P(\sE_K)}(1)|_X^{\otimes D}\right)\rightarrow H^0\left(X,\O_{\mathbb P(\sE_K)}(1)|_X^{\otimes D}\otimes\O_{X}/\sn_{\eta}^m\right)\right)
  \end{eqnarray*}
  for all $m\geqslant1$.
\end{prop}

\subsection{The volume of certain line bundles}
In this part, we will give a connection between the above invariant $\mathcal R_{\xi_j}(F_D)$ in Theorem \ref{semi-global determinant method} and the volume of certain line bundles.

\subsubsection{Definition of volume function}
In the first step, we will recall the definition of the volume of line bundles on projective varieties at \cite[Definition 2.2.31]{LazarsfeldI}. For more details about this notion, see \cite[\S 2.2.C]{LazarsfeldI}.

Let $X$ be a projective integral scheme of dimension $d$ over a field, and $L$ be a line bundle on $X$. We denote by $h^0\left(X, L\right)=\dim H^0\left(X,L\right)$ for simplicity. Then the \textit{volume} of the line bundle $L$ is defined to be the non-negative number
\begin{equation}\label{definition of volume}
  \vol\left(L\right)=\vol_X\left(L\right)=\limsup_{D\rightarrow\infty}\frac{h^0\left(X,L^{\otimes D}\right)}{D^d/d!}.
\end{equation}
Meanwhile, if $E$ is a Cartier divisor on $X$, we denote the volume by $\vol(E)$ or $\vol_X(E)$ for simplicity, or by passing $\O_X(E)$.

Let $NS(X)$ be the N\'eron-Severi group of $X$ (see \cite[Definition 1.1.15]{LazarsfeldI} for its definition). By \cite[Proposition 2.2.41]{LazarsfeldI}, the volume of a line bundle only depends on its class in N\'eron-Severi group. Let $NS(X)_{\mathbb R}=NS(X)\otimes_{\mathbb Z}\mathbb R$. By \cite[Corollary 2.2.45]{LazarsfeldI}, the volume function defined in \eqref{definition of volume} can be extended uniquely to a continuous function
\begin{equation}\label{volume over R}
  \vol:\; NS(X)_{\mathbb R}\rightarrow \mathbb R,
\end{equation}
where Cartier $\mathbb R$-divisors (see \cite[\S1.3.B]{LazarsfeldI} for its definition) are considered above.
\subsubsection{Dependence on the reduction}\label{depending on reduction}
We keep all the notations as above. Let $H$ be a Cartier divisor on $X$ given by a hyperplane section in $\mathbb P(\sE_K)$. Let $\eta_1,\eta_2\in X(K)$ be non-singular, and $\pi_1:\; \widetilde{X}_1\rightarrow X$ and $\pi_2:\; \widetilde{X}_2\rightarrow X$ be the blowing ups of $X$ at $\eta_1$ and $\eta_2$ respectively, with respect to the exceptional divisors $E_1\subset\widetilde{X}_1$ and $E_2\subset\widetilde{X}_2$. By Proposition \ref{R(E)-kernel} and Proposition \ref{kernel-blowup}, if two rational points of $\eta_1,\eta_2\in X(K)$ have the same non-singular specialization modulo a maximal ideal of $\O_K$ in the sense of Theorem \ref{semi-global determinant method}, then we have
\[h^0\left(\widetilde{X}_1, D\pi^*_1(H)-mE_1\right)=h^0\left(\widetilde{X}_2, D\pi^*_2(H)-mE_2\right)\]
for every $D,m\in\mathbb N$, which means it only depends on its specialization by the operation of Theorem \ref{semi-global determinant method}.
\subsubsection{Pseudo-effective thresholds}
By the fact stated in \S \ref{depending on reduction} above, we will introduce the following invariant.

\begin{defi}\label{I(X)}
  Let $X$ be a closed integral subscheme of $\mathbb P(\sE_K)$ over the number field $K$, $\eta\in X(K)$ whose specialization modulo $\p\in\spm\O_K$ is the non-singular point $\xi$ in the sense of Theorem \ref{semi-global determinant method}, $\pi:\;\widetilde X\rightarrow X$ be the blowing up at $\eta$, and $E\subset\widetilde X$ be its exceptional divisor. Let $H$ be a Cartier divisor on $X$ given by a hyperplane section in $\mathbb P(\sE_K)$. We define
  \[I_X(H,\xi)=\int_0^\infty\vol\left(\pi^*H-\lambda E\right)d\lambda,\]
  where the above volume function $\vol(\ndot)$ follows the extended definition introduced in \eqref{volume over R} over $\widetilde X$.
\end{defi}
\begin{rema}[History of $I_X(H,\xi)$]\label{history of I}
  To the author's knowledge, the invariant $I_X(H,\xi)$ given in Definition \ref{I(X)} is first introduced by Per Salberger in 2006 at a talk in Mathematical Science Research Institute (MSRI), Berkeley, USA. In \cite[\S4]{McKinnonRoth_2015}, D. Mckinnon and M. Roth also introduced this invariant for the research of Diophantine approximations over higher dimensional projective varieties, which is a generalization of Roth's theorem. In \cite{McKinnonRoth_2015}, they use the notation $\beta_x(L)=I_X(L,x)/\vol_X(L)$ for a closed point $x$ and an ample line bundle $L$.
\end{rema}

\subsection{The dominant term of $\mathcal R_{\xi_j}(F_D)$}
We keep all the above notations and conditions. We will give an asymptotic estimate of $\mathcal R_{\xi_j}(F_D)$ defined in \eqref{R(E)} by the invariant $I_X(H,\xi_j)$, where $j\in J$ is given in Theorem \ref{semi-global determinant method}.

\begin{theo}\label{R(E) by GIT}
  Let $X$ be a closed integral subscheme of $\mathbb P(\sE_K)$ of dimension $d$ and degree $\delta$ over a number field $K$. Let $F_D$ be the same as that in Theorem \ref{semi-global determinant method}, $\mathcal R_{\xi_j}(F_D)$ be defined in \eqref{R(E)}, where $j\in J$ and $\xi_j\in\mathscr X(\f_{\p_j})$ are same as those in Theorem \ref{semi-global determinant method}, and $H$ be a Cartier divisor on $X$ given by a hyperplane section in $\mathbb P(\sE_K)$. Then we have
  \[\mathcal R_{\xi_j}(F_D)=\frac{I_X(H,\xi_j)}{d!}D^{d+1}+O_{d,\delta}(D^d),\]
  where $I_X(H,\xi_j)$ is defined in Definition \ref{I(X)}.
\end{theo}
\begin{proof}
Let $\eta\in X(K)$ be a rational point whose reduction modulo $\p$ is $\xi$ in the sense of Theorem \ref{semi-global determinant method}, $\pi:\widetilde X\rightarrow X$ be the blowing up of $X$ at $\eta$, $E=\pi^{-1}(\eta)$ be the exceptional divisor of $\pi$. If denote $B=H^0(X,\O_X)$ and let $\dim B$ be the Krull dimension of the ring $B$, then by \cite[Lemma 2.1]{Poonen2004}, when $D\geqslant\dim B-1$, we have $F_D=H^0\left(\mathscr X,\O_{\mathbb P(\sE)}(1)|_{\mathscr X}^{\otimes D}\right)$ and $F_{D,K}=H^0\left(X,\O_{\mathbb P(\sE_K)}(1)|_X^{\otimes D}\right)$. Then by Proposition \ref{R(E)-kernel} and Proposition \ref{kernel-blowup}, we have
  \[\mathcal R_{\xi_j}(F_D)\sim\sum_{m=1}^\infty h^0\left(\widetilde X, D\pi^*H-mE\right)\]
when $D$ tends into infinite.

  Since $\vol(\pi^*H)=\vol(H)=\delta$, we have
  \[h^0\left(X,DH\right)=h^0\left(\widetilde X, D\pi^*H\right)=\frac{\delta}{d!}D^d+O_{d,\delta}(D^{d-1}).\]
  Meanwhile, if $m\geqslant1$, we have $0\leqslant h^0\left(\widetilde X, D\pi^*H-mE\right)\leqslant h^0\left(\widetilde X, D\pi^*H\right)$ and $\vol\left(\pi^*H-\frac{m}{D}E\right)\leqslant\vol\left(\pi^*H\right)=\delta$ for every $m\in\mathbb N$. Then by the definition of volume at \eqref{definition of volume}, when $m=1,\ldots,\left[\sqrt[d]{\delta}D\right]+1$, we have
  \begin{equation}\label{volume of DH-mE}
  h^0\left(\widetilde X, D\pi^*H-mE\right)=\frac{D^d}{d!}\vol\left(\pi^*H-\frac{m}{D}E\right)+O_{d,\delta}(D^{d-1}),
  \end{equation}
where $[s]$ denotes the largest integer smaller than $s\in \mathbb R$.

  By Proposition \ref{bound of finite sum}, we have
  \[\sum_{m=1}^\infty h^0\left(\widetilde X, D\pi^*H-mE\right)=\sum_{m=1}^{\left[\sqrt[d]{\delta}D\right]+1} h^0\left(\widetilde X, D\pi^*H-mE\right).\]

  By the estimate of remainder term in \eqref{volume of DH-mE} and Definition \ref{I(X)}, we have
  \begin{eqnarray*}
 \sum_{m=1}^{\left[\sqrt[d]{\delta}D\right]+1} h^0\left(\widetilde X, D\pi^*H-mE\right)&=&\frac{D^d}{d!}\sum_{m=1}^\infty\vol\left(\pi^*H-\frac{m}{D}E\right)+O_{d,\delta}(D^d)\\
  &=&\frac{I_X(H,\xi_j)}{d!}D^{d+1}+O_{d,\delta}(D^d),
\end{eqnarray*}
and we obtain the result.
\end{proof}
\begin{rema}
  By a result of Salberger announced in the MSRI lecture mentioned in Remark \ref{history of I} (see also \cite[Corollary 4.2]{McKinnonRoth_2015}), when $X\hookrightarrow \mathbb P(\sE_K)$ is of degree $\delta$ with respect to $\O_{\mathbb P(\sE_K)}(1)$, we have the following lower bound of $I_X(H,\xi)$ introduced in Definition \ref{I(X)}, which is
\[ I_X(H,\xi)\geqslant\frac{d\vol(H)}{d+1}\sqrt[d]{\frac{\vol(H)}{\mu_\eta(X)}}\geqslant\frac{d}{d+1}\epsilon_\eta(H)\vol(H),\]
where the reduction of $\eta\in X(K)$ modulo $\p\in\spm\O_K$ is $\xi$ in the sense of Theorem \ref{semi-global determinant method}, $\mu_\eta(X)$ is the multiplicity of $\eta$ in $X$, and $\epsilon_\eta(H)$ is the Seshadri constant of $H$ at $\eta$. For the application in this case, we have
\begin{equation}\label{lower bound of I_X}
I_X(H,\xi)\geqslant\frac{d\vol(H)}{d+1}\sqrt[d]{\frac{\vol(H)}{\mu_\eta(X)}}=\frac{d\delta^{1+\frac{1}{d}}}{(d+1)},
\end{equation}
since the point $\eta$ is regular in $X$, and $\vol(H)=H^d=\delta$ by definition. Then by Theorem \ref{R(E) by GIT}, we have
\[\mathcal R_{\xi_j}(F_D)\geqslant\frac{d\delta^{1+\frac{1}{d}}}{(d+1)!}D^{d+1}+O_{d,\delta}(D^d),\]
which is the same as that obtained in Proposition \ref{explicit lower bound of R(E)} and some other former results, for example, in \cite[Main Lemma 2.5]{Salberger07}.
\end{rema}
\section{The number of auxiliary hypersurfaces}
In this section, for a closed integral subscheme $X$ of $\mathbb P(\sE_K)$, we will give an upper bound of the number of hypersurfaces which cover $S(X;B)=\left\{\xi\in X(K)|\;H_K(\xi)\leqslant B\right\}$ but do not contain the generic point of $X$. The height function $H_K(\ndot)=\exp\left([K:\Q]h(\ndot)\right)$, and $h(\ndot)$ follows the definition \eqref{arakelov height} by Arakelov theory with respect to the Hermitian vector bundle $\overline{\sE}$ over $\spec\O_K$.
\subsection{Application of the asymptotic estimate of $\mathcal R_{\xi_j}(F_D)$}
Let $\overline{\sE}$ be a Hermitian vector bundle of rank $n+1$ over $\spec\O_K$, $X$ be a closed integral subscheme of $\mathbb P(\sE_K)$, and $\mathscr X$ be the Zariski closure of $X$ in $\mathbb P(\sE)$. Let $\p\in\spm\O_K$, and $\xi\in\mathscr X(\f_\p)$. We denote by $S(X;B,\xi)$ the subset of $S(X;B)$ whose reduction modulo $\p$ is $\xi$ in the sense of Theorem \ref{semi-global determinant method}.
\begin{lemm}\label{I_X(H,xi) same}
  We keep all the notations and conditions in Theorem \ref{semi-global determinant method}. If $\bigcap\limits_{j\in J}S(X;B,\xi_j)$ is not empty, then for every $j\in J$, all $\left\{I_X(H,\xi_j)\right\}_{j\in J}$ are equal, where $I_X(H,\xi_j)$ is defined in Definition \ref{I(X)}.
\end{lemm}
\begin{proof}
  By Proposition \ref{R(E)-kernel}, the invariant $I_X(H,\xi_j)$ only depends on its specialization. Then we obtain the assertion from Proposition \ref{kernel-blowup} directly.
\end{proof}
We keep all the notations and conditions in Lemma \ref{I_X(H,xi) same}, and we define
\begin{equation}\label{I_X(H,xi_J)}
  I_X(H,\xi_J)=I_X(H,\xi_j)
\end{equation}
for all $j\in J$. Then by the asymptotic estimate of $\mathcal R_\xi(F_D)$, we have the result below deduced from Theorem \ref{semi-global determinant method}.
\begin{theo}\label{semi-global determinant method mid}
  We keep all the notations in Theorem \ref{semi-global determinant method}. Let $\{\p_j\}_{j\in J}$ be a family of maximal ideals of $\O_K$ and $B,\epsilon>0$. For every $j\in J$, let $\xi_j\in \mathscr X(\f_{\p_j})$ be a regular point. Let $I_X(H,\xi_J)$ be defined in \eqref{I_X(H,xi_J)} (by Lemma \ref{I_X(H,xi) same} it is well defined). If the inequality
  \begin{equation}\label{hypothesis of step 1}
\sum_{j\in J}\log N(\p_j)\geqslant (1+\epsilon)\left(\log B +[K:\Q]\frac{\log\left((n+1)(d+1)\right)}{2}\right)\frac{\delta}{I_X(H,\xi_J)}
  \end{equation}
  is verified, then there exists a hypersurface of degree $O_{d,\delta,\epsilon}(1)$ in $\mathbb P(\sE_K)$, which contains the set $\bigcap\limits_{j\in J}S(X;B,\xi_j)$ but do not contain the generic point of $X$.
\end{theo}
\begin{proof}
  We only need to prove the assertion for the case when $\bigcap\limits_{j\in J}S(X;B,\xi_j)\neq\emptyset$. Let $D\in\mathbb N^+$, and we suppose that such there does not exist such a hypersurface of degree $D$. By Theorem \ref{semi-global determinant method}, we have
  \begin{equation}\label{first step of determinant method}
    \frac{\log B}{[K:\Q]}\geqslant\frac{\wmu(\F_D)}{D}-\frac{\log r_1(D)}{2D}+\sum_{j\in J}\frac{\mathcal R_{\xi_j}(F_D)}{Dr_1(D)}\frac{\log N(\p_j)}{[K:\Q]}.
  \end{equation}

 For every $j\in J$, $\xi_j$ is regular in $\mathscr X$, and we have
  \[r_1(D)=\frac{\delta}{d!}D^d+O_{d,\delta}(D^{d-1}).\]
  Then we apply Theorem \ref{R(E) by GIT} by combining the above two facts, and we obtain that there exists a constant $C(d,\delta)$ depending on $d$ and $\delta$, such that
  \[\frac{\mathcal R_{\xi_j}(F_D)}{Dr_1(D)}\geqslant \frac{I_X(H,\xi_J)}{\delta}+\frac{C(d,\delta)}{D}\]
  is verified for each $D\geqslant1$ and $j\in J$. By \cite[\S1.2]{Chardin89}, we have
  \[r_1(D)\leqslant\delta{D+d\choose D}\leqslant\delta(d+1)^D.\]
  We combine the above arguments and the trivial lower bound of $\wmu(\F_D)$ introduced at \eqref{trivial lower bound of F_D}. From the inequality \eqref{first step of determinant method}, we have
  \[\frac{\log B}{[K:\Q]}\geqslant-\frac{1}{2}\log(n+1)-\frac{\log\delta}{2D}-\frac{1}{2}\log(d+1)+\left(\frac{I_X(H,\xi_J)}{\delta}+\frac{C(d,\delta)}{D}\right)\sum_{j\in J}\frac{\log N(\p_j)}{[K:\Q]},\]
  and we obtain
  \begin{eqnarray*}
    & &\left(\frac{I_X(H,\xi_J)}{\delta}\sum_{j\in J}\frac{\log N(\p_j)}{[K:\Q]}-\frac{\log B}{[K:\Q]}-\frac{1}{2}\log(n+1)-\frac{1}{2}\log(d+1)\right)D\\
    &\leqslant&\left(-\frac{\log\delta}{2}+C(d,\delta)\right)\sum_{j\in J}\frac{\log N(\p_j)}{[K:\Q]}.
  \end{eqnarray*}
  By the hypothesis \eqref{hypothesis of step 1}, the left hand side of the above inequality is larger than or equal to
  \[\frac{\epsilon}{1+\epsilon}\cdot\frac{I_X(H,\xi_J)}{\delta}\sum_{j\in J}\frac{\log N(\p_j)}{[K:\Q]}D,\]
  which implies
  \[D\leqslant(\epsilon^{-1}+1)\frac{\delta}{I_X(H,\xi_J)}\left(-\frac{\log\delta}{2}+C(d,\delta)\right).\]
  By \eqref{lower bound of I_X} and the fact that all $\xi_j$ is regular in $\mathscr X_{\f_{\p_j}}$ for each $j\in J$, there exists a lower bound of $I_X(H,\xi_J)$ which only depends on the $d$ and $\delta$. Then we obtain a contradiction, which terminates the proof.
\end{proof}
The following result can be considered as a generalization of \cite[Main Lemma 16.3.1]{Salberger2015}.
\begin{coro}\label{semi-global determinant method 2}
  We keep all the notations and conditions in Theorem \ref{semi-global determinant method mid}. Let
  \begin{equation}\label{inf of I_X}
    I_X(H)=\inf_{\begin{subarray}{c}\eta\in S(X^\mathrm{reg};B)\end{subarray}}\{I_X(H,\eta)\}.
  \end{equation}
  If the inequality
    \begin{equation*}
    \sum_{j\in J}\log N(\p_j)\geqslant (1+\epsilon)\left(\log B +\frac{1}{2}[K:\Q]\log\left((n+1)(d+1)\right)\right)\frac{\delta}{I_X(H)}
  \end{equation*}
  is verified, then there exists a hypersurface of degree $O_{n,\delta,\epsilon}(1)$ in $\mathbb P(\sE_K)$, which contains $\bigcap\limits_{j\in J}S(X;B,\xi_j)$ but does not contain the generic point of $X$.
\end{coro}
\begin{proof}
  By definition \eqref{inf of I_X}, we have
  \[\frac{\delta}{I_X(H)}\geqslant\frac{\delta}{I_X(H,\xi_J)},\]
  where $I_X(H,\xi_J)$ is defined in the assertion of Theorem \ref{semi-global determinant method mid}. Then we obtain this result from \eqref{hypothesis of step 1} in Theorem \ref{semi-global determinant method mid} directly.
\end{proof}

\subsection{Bertrand's postulate of number fields}
In order to apply Theorem \ref{semi-global determinant method mid} and Corollary \ref{semi-global determinant method 2}, we need some estimate about the distribution of prime ideals of rings of algebraic integers. In fact, we need an analogue of Bertrand's postulate for the case of number fields, which follows.
\begin{lemm}\label{Bertrand's postulate}
  Let $K$ be a number field, and $\O_K$ be the ring of integers of $K$. There exists a constant $\alpha(K)\geqslant2$ depending on $K$, such that for all number $N_0\geqslant1$, there exists at least one maximal ideal $\p$ of $\O_K$, such that $N_0<N(\p)\leqslant\alpha(K)N_0$.
\end{lemm}
We refer to \cite{Lagarias-Odlyzko} or \cite[Th\'eor\`eme 2]{Serre-postulat} for a proof by admitting the Generalized Riemann Hypothesis, and to \cite[Th\'eor\`eme 1.7]{Winckler_these} without admitting it.
\subsection{Complexity of the singular locus}
Let $\overline{\sE}$ be a Hermitian vector bundle of rank $n+1$ over $\spec\O_K$, $X$ be a closed integral subscheme of $\mathbb P(\sE_K)$ of degree $\delta$ and dimension $d$. In order to give an upper bound of the number of auxiliary hypersurfaces which cover $S(X;B)$ but do not contain the generic point of $X$, we divide $S(X;B)$ into two part: the part of regular points and the part of singular points. In this part, we will deal with the singular part $S(X^{\mathrm{sing}};B)$.

By \cite[Theorem 3.10]{Chen1} (see also \cite[\S 2.6]{Chen2}), we have the following control to the complexity of the singular locus.
\begin{prop}\label{covering singular locus}
    Let $\overline{\sE}$ be a Hermitian vector bundle of rank $n+1$ over $\spec\O_K$, and $X$ be a closed integral subscheme of $\mathbb P(\sE_K)$, which is of degree $\delta$ and of dimension $d$. Then there exists a hypersurface of degree $(\delta-1)(n-d)$ in $\mathbb P(\sE_K)$ which covers $S(X^{\mathrm{sing}};B)$ but do not contain the generic point of $X$.
\end{prop}
%\begin{rema}
 % A kind of height of the hypersurface which covers $S(X^{\mathrm{sing}};B)$ determined in Proposition \ref{covering singular locus} is bounded by
  %\[-(n-d)h_{\overline{\O(1)}}(X)-C_3,\]
 % where the constant $C_3$ is defined at \eqref{constant C_3}, and $)h_{\overline{\O(1)}}(X)$ is introduced at Remark \eqref{definition of arakelov height}. This kind of height function is introduced at \cite[\S 3.1]{Chen1}, and it is just the slope of the Hermitian vector bundle $\overline I_X$ over $\spec\O_K$ there. By \cite[Proposition 3.6]{Chen1}, we can compare this height function with the one defined by the arithmetic intersection theory (see Remark \ref{definition of arakelov height}) uniformly depending on its degree and dimension. In this paper, we will not use this height function, so we do not plan to introduce its precise definition in this article.
%\end{rema}
\subsection{Control of regular reductions}
Let $\p\in\spm\O_K$, $S(X^{\mathrm{reg}};B)$ be the subset of $S(X;B)$ consisting of regular points, and $S(X;B,\xi)$ be the subset of $S(X;B)$ whose reduction modulo $\p$ is $\xi$, where the operation modulo $\p$ follows the sense of Theorem \ref{semi-global determinant method}. We denote
\begin{equation}\label{S(X;B,p)}
S(X^{\mathrm{reg}};B,\p)=\bigcup_{\begin{subarray}{c}\xi\in\mathscr X(\f_\p)\\ \mu_\xi\left(\mathscr X\right)=1\end{subarray}}S(X;B,\xi).
\end{equation}
In other words, $S(X^{\mathrm{reg}};B,\p)$ is the subset of $S(X^{\mathrm{reg}};B)$ with regular reduction modulo $\p$.

Next, we will refer a result that $S(X^{\mathrm{reg}};B)$ can be covered by some $S(X^{\mathrm{reg}};B,\p)$ for some particular $\p\in\spm\O_K$. For this aim, we introduce the following constants original from \cite[Notation 19]{Chen2}. Let
\begin{eqnarray*}
  C_1&=&(d+2)\wmu_{\max}\left(\sym^\delta\left(\overline{\sE}^\vee\right)\right)+\frac{1}{2}(d+2)\log\rg\left(\sym^\delta \sE\right)\\
  & &\;+\frac{\delta}{2}\log\left((d+2)(n-d)\right)+\frac{\delta}{2}(d+1)\log(n+1),
\end{eqnarray*}
\begin{equation*}
  C_2=\frac{r}{2}\log \rg\left(\sym^\delta\sE\right)+\frac{1}{2}\log \rg\left(\wedge^{n-d}\sE\right)+\log\sqrt{(n-d)!}+(n-d)\log \delta,
\end{equation*}
and
\begin{equation}\label{constant C_3}
  C_3=(n-d)C_1+C_2.
\end{equation}
The above constant $C_1$ is original from \cite[(21)]{Chen1}, and $C_2$ is from \cite[Remark 3.9]{Chen1}. The constant $C_3$ firstly appeared at \cite[Theorem 3.10]{Chen1}, and we have
\begin{equation}\label{estimate of C_3}
  C_3\ll_{n,d}\delta.
\end{equation}
By the above notations, we state the following result.
\begin{lemm}[\cite{Chen2}, Lemma 4.1]\label{upper bound of regular reduction}
  Let $N_0>0$ be a real number and $r$ be the integral part of
  \begin{equation}\label{number of places for regular reductions}
    \frac{(n-d)(\delta-1)\log B+\left((n-d)h_{\overline{\O_{\mathbb P(\sE)}(1)}}(X)+C_3\right)[K:\Q]}{\log N_0}+1,
  \end{equation}
  where the constant $C_3$ is defined in \eqref{constant C_3}, and the height $h_{\overline{\O_{\mathbb P(\sE)}(1)}}(X)$ is defined in \eqref{definition of arakelov height}. If $\p_1,\ldots, \p_r$ are distinct maximal ideals of $\O_K$ such that $N(\p_i)>N_0$ is verified for every $i=1,\ldots,r$, then
    \[S(X^{\mathrm{reg}};B)=\bigcup_{i=1}^rS(X^{\mathrm{reg}};B,\p_i),\]
    where every $S(X^{\mathrm{reg}};B,\p_i)$ is defined in \eqref{S(X;B,p)}.
\end{lemm}
\subsection{An upper bound of the number of auxiliary hypersurfaces}
In this part, we will estimate the number of auxiliary hypersurfaces which cover $S(X;B)$ but do not contain the generic point of $X$. In fact, by Proposition \ref{covering singular locus}, we only need to consider the regular part $S(X^{\mathrm{reg}};B)$.

By \cite[Theorem 4.8]{Chen1} and \cite[Proposition 2.12]{Chen1}, the rational points with small height in $X$ can be covered by one hypersurface of degree $O_{n}(\delta)$ not containing the generic point of $X$, where the "small" height means that the bound $B$ is small compared with the height of $X$. We will use the above argument to deal with the points with small height and the method of Theorem \ref{semi-global determinant method} and Proposition \ref{semi-global determinant method 2} to deal with the regular points with large height, and combine it with Lemma \ref{upper bound of regular reduction}.
\begin{theo}\label{number of hypersurfaces}
Let $K$ be a number field and $\O_K$ be its ring of integers. Let $\overline{\sE}$ be a Hermitian vector bundle of rank $n+1$ over $\spec\O_K$, $X$ be a closed integral subscheme of $\mathbb P(\sE_K)$ of dimension $d$ and degree $\delta$, and $\epsilon>0$ be an arbitrary real number. Then there exists an explicit constant $C_4(\epsilon, \delta, n,d, K)$, such that for every $B\geqslant e^\epsilon$, the set $S(X;B)$ can be covered by no more than $C_4(\epsilon, \delta, n,d, K)B^{\frac{(1+\epsilon)d\delta}{I_X(H)}}$ hypersurfaces with degree of $O_{n,\delta,\epsilon}(1)$ which do not contain the generic point of $X$, where $I_X(H)$ is defined in \eqref{inf of I_X}.
\end{theo}
\begin{proof}
We divide this proof into two parts, the case of varieties with large height and the case of varieties with small height.

\textbf{Part 1. Case of large height varieties. - }Suppose that the inequality
\[h_{\overline{\O_{\mathbb P(\sE)}(1)}}(X)>\frac{(2d+2)^{d+1}}{d!}\delta\left(\frac{\log B}{[K:\Q]}+\frac{3}{2}\log(n+1)+2^d\right)\]
 is verified, where $h_{\overline{\O_{\mathbb P(\sE)}(1)}}(X)$ is defined in \eqref{definition of arakelov height}. Then by \cite[Theorem 4.8]{Chen1} and \cite[Proposition 2.12]{Chen1} (see also \S 2.1 and \S 2.3 of \cite{Chen2}), there exists a hypersurface in $\mathbb P(\sE_K)$ of degree $2(n-d)(\delta-1)+d+2$ which covers $S(X;B)$ but does not contain the generic point of $X$.

\textbf{Part 2. Case of small height varieties. - } Now we suppose that the inequality
  \[h_{\overline{\O_{\mathbb P(\sE)}(1)}}(X)\leqslant\frac{(2d+2)^{d+1}}{d!}\delta\left(\frac{\log B}{[K:\Q]}+\frac{3}{2}\log(n+1)+2^d\right)\]
  is verified. Let
  \[\log N_0=(1+\epsilon)\left(\log B +\frac{1}{2}[K:\Q]\left(\log(n+1)+\log(d+1)\right)\right)\frac{\delta}{I_X(H)},\]
  and $r$ be the positive integer defined in \eqref{number of places for regular reductions} of Lemma \ref{upper bound of regular reduction}. In this case, we have
  \[r\leqslant \frac{A_1\log B+A_2}{\log N_0}+1,\]
  where the constants are
    \[A_1=(n-d)(\delta-1)+\frac{(2d+2)^{d+1}}{d!}(n-d)\delta,\]
  and
  \[A_2=[K:\Q]\left(C_3+\frac{(2d+2)^{d+1}}{d!}\delta\left(\log(n+1)+\frac{1}{2}\log(d+1)+2^d\right)\right)\]
  with the constant $C_3$ is defined in \eqref{constant C_3}. By the assumption $\log B\geqslant\epsilon$, we obtain $r\leqslant A_3$, where
  \[A_3=\frac{I_X(H)}{\delta}\left(A_1+\epsilon^{-1}A_2\right)+1.\]

  By Bertrand's postulate (cf. Lemma \ref{Bertrand's postulate}), there exists a family of maximal ideals $\p_1,\ldots,\p_r$ of $\O_K$, such that
  \[\alpha(K)^{i-1}N_0\leqslant N(\p_i)\leqslant\alpha(K)^iN_0\]
  for every $i=1,\ldots,r$, where the constant $\alpha(K)\geqslant2$ depends only on the number field $K$.

  For each $\p_i$, we have
  \[\#\mathscr X\left(\f_{\p_i}\right)\leqslant\delta\left(N(\p_i)^d+\cdots+1\right)\leqslant\delta(d+1)N(\p_i)^d\leqslant\delta(d+1)\alpha(K)^{di}N_0^d,\]
  and then we obtain the following upper bound of the number of auxiliary hypersurfaces which cover $S_1(X;B)$ but do not cover the generic point of $X$. The upper bound mentioned above is
  \begin{eqnarray*}
    \sum_{i=1}^r\#\mathscr X(\f_{\p_i})&\leqslant&\delta(d+1)N_0^d\sum_{i=1}^r\alpha(K)^{di}\\
    &=&\delta(d+1)N_0^d\frac{\alpha(K)^d(\alpha(K)^{rd}-1)}{\alpha(K)^d-1}\\
    &\leqslant& C'_4B^{\frac{(1+\epsilon)d\delta}{I_X(H)}},
  \end{eqnarray*}
where the constant
\begin{eqnarray}\label{constant C''_4}
C'_4&=&\delta(d+1)\frac{\alpha(K)^d(\alpha(K)^{A_3d}-1)}{\alpha(K)^d-1}\left((d+1)(n+1)\right)^{\frac{(1+\epsilon)[K:\Q]d\delta}{2I_X(H)}}\nonumber\\
&\leqslant&\delta(d+1)\frac{\alpha(K)^d(\alpha(K)^{A_3d}-1)}{\alpha(K)^d-1}\left((d+1)(n+1)\right)^{\frac{1}{2}(1+\epsilon)[K:\Q](d+1)\delta^{-\frac{1}{d}}}\nonumber\\
&=:&C''_4(\epsilon, \delta, n,d, K).
\end{eqnarray}
In the above inequality, the second line is from the lower bound of $I_X(H)$ provided at \cite[Corollary 4.2]{McKinnonRoth_2015} (see \eqref{lower bound of I_X} for this lower bound in our application) and the definition of $I_X(H)$ at \eqref{inf of I_X}. Then we obtain the assertion by Corollary \ref{semi-global determinant method 2}.

\textbf{Conclusion. - }By the above argument, we obtain the result after combining it with Proposition \ref{covering singular locus}, where we choose the constant $C_4(\epsilon, \delta, n,d, K)=C''_4(\epsilon, \delta, n,d, K)+1$ introduced at \eqref{constant C''_4}.
\end{proof}
\begin{rema}
  In the proof of Theorem \ref{number of hypersurfaces}, by the fact that $A_1\ll_{n,d}\delta$ and $A_2\ll_{n,d}\delta$, we have $A_3\ll_{n,d,\epsilon}\delta^{1+\frac{1}{d}}$, we obtain
  \[\log C_4(\epsilon, \delta, n,d, K)\ll_{n,K,\epsilon}\delta^{1+\frac{1}{d}},\]
  since we have $1\leqslant d\leqslant n-1$. But the above estimate of $C_4(\epsilon, \delta, n,d, K)$ is valueless unless we are able to obtain an explicit estimate of the degree of auxiliary hypersurfaces.
\end{rema}
\begin{rema}
  In Theorem \ref{number of hypersurfaces}, we do not give an explicit upper bound of the degree of auxiliary hypersurfaces. The main obstruction is that in Theorem \ref{R(E) by GIT}, when we estimate $\mathcal R_{\xi_j}(F_D)$, until the author's knowledge, we cannot find an explicit lower bound of $\dim H^0\left(X,\mathscr L^{\otimes m}\right)$ for arbitrary line bundle $\mathscr L$. If $\mathscr L$ is ample, see \cite[Page 92]{Kollar07} for such an explicit lower bound. So by the strategy of this article, we are not able to control the dependence of $S(X;B)$ on the degree of $X$ due to the limit of the author's ability.
\end{rema}
\backmatter

\bibliography{liu}
\bibliographystyle{smfplain}

\end{document}